\documentclass[12pt]{amsart}
\usepackage[margin=1.1in]{geometry}

\usepackage{epic, eepic, amsfonts, latexsym, amssymb, graphicx,
multicol, mathrsfs, color, amscd, amsxtra, verbatim, paralist,
xspace, url, stmaryrd, amsmath, enumitem,
bbold}
\usepackage{fmtcount}

\usepackage{times}

\usepackage{multirow, tikz}

\usepackage[all]{xy}

\usepackage[colorlinks, linkcolor=blue, citecolor=magenta, urlcolor=cyan]{hyperref}

\newtheorem{lem}{Lemma}[section]
\newtheorem{thm}[lem]{Theorem}

\newtheorem{prop}[lem]{Proposition}

\newtheorem{cor}[lem]{Corollary}

\theoremstyle{definition}
\newtheorem{definition}[lem]{Definition}
\newtheorem{example}[lem]{Example}

\theoremstyle{remark}
\newtheorem{remark}[lem]{Remark}

\numberwithin{equation}{section}

  \newcommand\cC{\mathcal{C}}  \newcommand\cL{\mathcal{L}}\newcommand\cM{\mathcal{M}}\newcommand\cO{\mathcal{O}}\newcommand\cT{\mathcal{T}}\newcommand\cU{\mathcal{U}}\newcommand\cV{\mathcal{V}}\newcommand\cY{\mathcal{Y}}\newcommand\cZ{\mathcal{Z}}

\renewcommand\AA{\mathbb{A}}\newcommand\GG{\mathbb{G}}\newcommand\HH{\mathbb{H}}\newcommand\PP{\mathbb{P}}\newcommand\QQ{\mathbb{Q}}
\newcommand\ZZ{\mathbb{Z}}

\newcommand\arr{\ifinner\to\else\longrightarrow\fi}

\newcommand\hookarr{\hookrightarrow}

\renewcommand{\setminus}{\smallsetminus}

\renewcommand{\ss}{\operatorname{ss}}

\newcommand{\Hilb}{\operatorname{Hilb}}
\newcommand{\Pic}{\operatorname{Pic}}
\newcommand{\Def}{\operatorname{Def}}
\newcommand{\Cr}{\operatorname{Cr}}

\newcommand{\Proj}{\operatorname{Proj}}

\newcommand{\Aut}{\operatorname{Aut}}

\newcommand{\oh}{\cO}

\newcommand{\Spec}{\operatorname{Spec}}

\newcommand{\tensor} {\otimes}

\renewcommand{\tilde}{\widetilde}

\newcommand{\Spf}{\operatorname{Spf}}

\renewcommand{\bar}{\overline}

\newcommand{\PGL}{\operatorname{PGL}}
\newcommand{\SL}{\operatorname{SL}}

\newcommand{\dual}{\vee}
\renewcommand{\hat}{\widehat}

\renewcommand{\L}{\mathcal{L}}

\def\nb{\nobreakdash}
\def\F{\mathcal{F}}
\def\N{\mathrm{N}}

\renewcommand{\HH}{\mathrm{H}}
\DeclareMathOperator\spec{Spec}
\def\eps{\varepsilon}

\newcommand{\gitq}{/\hspace{-0.25pc}/}
\newcommand\Mg[1]{\overline{\mathcal{M}}_{#1}}
\newcommand\M{\overline{M}}
\newcommand\ra{\rightarrow}

\renewcommand{\L}{\mathcal{L}}
\def\nb{\nobreak}
\def\F{\mathcal{F}}
\def\N{\mathrm{N}}
\def\co{\colon\thinspace} 

\newcommand\X{\mathcal{X}}

\begin{document}

\title[Singularities with $\GG_m$-action and the log MMP for $\overline{M}_g$]{Singularities with $\GG_m$-action and\\ the log minimal model program for $\overline{M}_g$}

\author[Alper]{Jarod Alper}
\author[Fedorchuk]{Maksym Fedorchuk}
\author[Smyth]{David Smyth}

\address[Alper]{Department of Mathematics\\
Columbia University\\
2990 Broadway\\
New York, NY 10027}
\email{jarod@math.columbia.edu}

\address[Fedorchuk]{Department of Mathematics\\
Columbia University\\
2990 Broadway\\
New York, NY 10027}
\email{mfedorch@math.columbia.edu}

\address[Smyth]{Department of Mathematics\\
Harvard University\\
1 Oxford Street\\
Cambridge, MA 01238}
\email{smyth@math.harvard.edu}


\begin{abstract}
We give a precise formulation of the modularity 
principle for the log canonical models 
$\bar{M}_g(\alpha) := \Proj \bigoplus_{d \ge 0} \HH^0(\bar{\cM}_g, d( K_{\bar{\cM}_g} + \alpha \delta))$.  
Assuming the modularity principle holds, we develop and compare two methods for determining the critical 
$\alpha$-values at which a singularity or complete curve with 
$\mathbb{G}_m$-action arises in the modular interpretation of $\bar{M}_g(\alpha)$. The first method involves 
a new invariant of curve singularities with $\mathbb{G}_m$-action, constructed via the characters of the 
induced $\mathbb{G}_m$-action on spaces of pluricanonical forms. The second method involves intersection 
theory on the variety of stable limits of a singular curve.  We compute the expected $\alpha$-values for large 
classes of singular curves, including curves with ADE, toric, and 
unibranch Gorenstein 
singularities, as well as for ribbons, and show that the two methods yield identical predictions. We use these 
results to give a conjectural outline of the log MMP for $\bar{M}_g$.
\end{abstract}

\maketitle

\setcounter{tocdepth}{1}
\tableofcontents

\section{Introduction} \label{S:introduction}
In an effort to understand the canonical model of $\bar{M}_g$, Hassett and Keel initiated a program to give modular interpretations to the log canonical models
\[
\bar{M}_g(\alpha) := \Proj \bigoplus_{d \ge 0} \HH^0(\bar{\cM}_g, \lfloor d( K_{\bar{\cM}_g} + \alpha \delta) \rfloor),
\]
for all $\alpha \in [0,1] \cap \QQ$ such that $K_{\bar{\cM}_g} + \alpha \delta$ is effective; see \cite{Hgenus2, hassett-hyeon_contraction,hassett-hyeon_flip, hyeon-lee_genus3}.  The assertion that these log canonical models should admit modular interpretations, which is implicit in the work of Hassett, Hyeon, and Keel, can be formulated precisely as follows:

\vspace{1pc}

\noindent
\textbf{Modularity principle for the log MMP for $\M_g$. } {\it Let $\cU_{g}$ be the stack 
of all complete connected Gorenstein curves $C$ of arithmetic genus $g$ with $\omega_C$ ample.
For $\alpha \in [0,1] \cap \QQ$ such that $K_{\bar{\cM}_g} + \alpha \delta$ is effective, there exists an open 
substack $\bar{\cM}_g(\alpha) \subseteq \cU_g$, and a map 
$\phi\co\bar{\cM}_g(\alpha) \rightarrow \bar{M}_g(\alpha)$ such that
$\phi$ is cohomologically-affine and $\phi_*\cO_{ \bar{\cM}_g(\alpha)}=\cO_{\bar{M}_g(\alpha)}$. 
Equivalently, the log canonical model $\M_{g}(\alpha)$ is a good moduli space for $\bar{\cM}_{g}(\alpha)$.
} 

\vspace{1pc}

\par
\noindent
We refer to \cite{alper_good_arxiv} for 
a discussion of the essential properties of good moduli spaces, which may be thought of as best-possible approximations to a coarse moduli space in cases where the existence of a coarse moduli space is precluded by non-separatedness of the moduli stack. Hassett, Hyeon, and Lee have verified the modularity principle for the log MMP for $\bar{M}_{g}$ for all $\alpha$ when $g=2,3$ \cite{Hgenus2, hyeon-lee_genus3}, and for $\alpha > \frac{7}{10}-\epsilon$ in arbitrary genus \cite{hassett-hyeon_contraction,hassett-hyeon_flip}. In exploring possible extensions of their work, it is natural to consider the following question: Assuming the modularity principle holds, what curves should appear in the stacks $\bar{\cM}_{g}(\alpha)$? How can we tell at which $\alpha$-value a given singular curve should appear? In this paper, we develop two methods for answering these questions, at least for curves with a $\mathbb{G}_m$-action, and show that the two methods give identical predictions.

To explain the first method, consider a complete 
curve $C$ with a $\GG_m$-action 
$\eta\co \GG_m \to \Aut(C)$ and an isolated singularity at a point $p \in C$.  If $\cL$ is a line bundle on 
$\bar{\cM}_g$ that 
extends to a neighborhood of $[C]$ in the stack of all curves, then there is an induced action of $\GG_m$ on 
the fiber of $\cL$ over $[C]$ given by a character $\chi_{\cL}(C,\eta)$.  The key observation connecting 
characters with the modularity principle is: 
If a curve $C$ is to appear in $\bar{\cM}_g(\alpha)$ for some $\alpha$, then the character of 
$K_{\bar{\cM}_{g}(\alpha)}+ \alpha \delta = 13 \lambda - (2-\alpha) \delta$ 
is necessarily trivial since the line bundle descends to $\bar{M}_g(\alpha)$; this is the essence of 
Proposition \ref{C:CriticalAlpha}.
We compute the characters of generators of $\text{Pic}(\bar{\cM}_g)$ for a large class of singular curves with $\GG_m$-action.  In particular, we calculate the characters for curves with ADE singularities, planar toric singularities, and unibranch Gorenstein singularities, as well as for ribbons;
we collect our results in Tables \ref{T:table-characters} and \ref{T:table-dangling}. As a consequence, we predict the precise $\alpha$-values at which curves with these singularities arise in the modular interpretations of $\bar{\cM}_g(\alpha)$; see Table \ref{table-predictions}. This is our first main result.

Our second method for predicting $\alpha$-values is based on the following observation: If a locus $\cT\subseteq \overline{\cM}_g$ is covered by $(K_{\Mg{g}}+\alpha\delta)$-negative curves, i.e. curves on 
which $K_{\Mg{g}}+\alpha\delta$ has negative degree, then $\cT $ falls in the stable base 
locus of $K_{\Mg{g}}+\alpha\delta$ and thus is flipped by the rational map $\M_g \dashrightarrow \bar{M}_g(\alpha)$.  If $\cT = \cT_C$ is the variety of stable curves arising from stable reduction of a singular curve $C$, then
$C$ appears in a modular interpretation of $\bar{\cM}_g(\alpha)$ for those $\alpha$ such that
$\cT_C$ is covered by $(K_{\Mg{g}} + \alpha \delta)$-negative curves. In Section \ref{S:intersection-theory}, we compute these anticipated $\alpha$-values for toric singularities using degeneration and intersection theory techniques. 
Comparing with Table \ref{T:table-characters}, we observe that the $\alpha$-values obtained by character theory and intersection theory are the same.  The fact that the two techniques yield the same $\alpha$-values is not merely coincidental: In Theorem \ref{theorem-character-intersection}, we prove a general theorem which provides a formal relationship between the characters and the intersection numbers.

Because these two heuristics give such a useful guide for defining the moduli functors $\Mg{g}(\alpha)$, 
we expect them to play an important role in verifying the modularity principle for the log MMP for 
$\M_{g}$ for $\alpha<7/10$.

\subsection{Notation} \label{section-notation}
 We work over an algebraically closed field $k$ of characteristic $0$.
Let $\cU_g$ be the stack of all complete connected Gorenstein curves $C$ with $\omega_C$ ample.  Let $\pi\co \cC_g \arr \cU_g$ be the universal curve.  Denote by $\omega_{\cC_g / \cU_g}$ the relative dualizing sheaf on $\cC_g$.  
The following line bundles are defined on all of
$\cU_g$:
\begin{align*}
\lambda = \lambda_1	& := \det \pi_* (\omega_{\cC_g / \cU_g}), \\
\lambda_m			& := \det \pi_* (\omega_{\cC_g / \cU_g}^{m}).
\end{align*}
The following divisor classes are defined on any open 
substack 
 of $\cU_g$ that satisfies Serre's condition $S_2$ and whose locus of 
worse-than-nodal curves is of codimension at least $2$:
\begin{align*}
\kappa	& = \pi_* (c_1^2(\omega_{\cC_g / \cU_g})), \\
K					& = \text{the canonical divisor class}, \\
\delta_0 = \delta_{\text{irr}}&= \text{the divisor of irreducible singular curves}, \\
\delta = \delta_{0} + \delta_{\text{red}} & = \delta_0 + \delta_1 + \cdots + \delta_{\lfloor g/2 \rfloor}.
\end{align*}
We can define $K$ simply as $K=13\lambda - 2\delta$; see below for a more intrinsic definition.  Furthermore, we have the following relations on this open substack:
\begin{equation} \label{relations}
\begin{aligned}
\lambda_2	 &= 13 \lambda - \delta = K + \delta, \\   
\kappa &= 12 \lambda - \delta = - \lambda + \lambda_2 = \frac{12}{13}\bigl(K + \frac{11}{12} \delta\bigr).  
\end{aligned}
\end{equation}
We define the \emph{slope} of a divisor $s \lambda - \delta$ to be $s$ and the \emph{$\alpha$-value} of 
a divisor $K + \alpha \delta$ to be $\alpha$.  
In particular, the slope of $K+ \alpha \delta$ is $13/(2-\alpha)$ and the $\alpha$-value of $s\lambda -  \delta$ is 
$2- 13/s$.

If $B\rightarrow \Mg{g}$ is a complete curve, the {\em slope} of $B$ is defined to be 
$(\delta\cdot B)/(\lambda\cdot B)$.

\subsection{Defining the canonical divisor $K$}  \label{section-discussion-K}
Let $\cM$ be the smooth locus of $\cU_g$.  
Consider the cotangent complex $L_{\cM}$ of $\cM$ which can be described explicitly as
follows. Choose a quotient stack presentation $\cM=[M / G]$; e.g., we
may take $M$ to be a Hilbert scheme and $G$ to be a group of
automorphisms of a projective space.  Then the cotangent complex is given by:
$$L_{\cM}\co  (\Omega_M \xrightarrow{\alpha} \mathfrak{g}^{\dual} \otimes \cO_M),$$
where $\Omega_M$ inherits its natural $G$-linearization and
$\mathfrak{g}$ is the adjoint representation; the morphism $\alpha$ is
the pullback of the natural map 
$\text{pr}_2^{\,*} \Omega_M \to \Omega_{G \times M} = \text{pr}_2^{\,*} (\Omega_G)$ along the identity section.

Then the canonical line bundle is defined as
$$K_{\cM} := \det L_{\cM} = K_M \tensor (\mathfrak{g} \tensor \oh_M).$$

\begin{remark} One can check that the Grothendieck-Riemann-Roch
calculation implies that
$$K_{\cM} =13 \lambda - 2 \delta$$
whenever $K_{\cM}$, $\lambda$ and $\delta$ are all defined. 
\end{remark}

\subsection*{Acknowledgements} 
We thank Anand Deopurkar, David Hyeon, and Fred van der Wyck for stimulating discussions. 
We also thank David Hyeon for sharing an early version of his preprint \cite{hyeon_predictions}.

\section{$\alpha$-invariants of curve singularities}
\label{S:character-theory}
To begin, recall that if $[C] \in \cU_g$ is any point and $\cL$ is a line bundle defined in a neighborhood of $[C]$, then the natural action of $\Aut(C)$ on the fiber $\cL\vert_{[C]}$ induces a character $\Aut(C) \rightarrow \mathbb{G}_m$. If $\eta\co \mathbb{G}_m \rightarrow \Aut(C)$ is any one-parameter subgroup, then there is an induced character $\mathbb{G}_m \rightarrow \mathbb{G}_m$ which is necessarily of the form $z \mapsto z^{n}$ for some integer $n \in \mathbb{Z}$. Given a curve $[C]$, a one-parameter subgroup $\eta\co  \mathbb{G}_m \rightarrow \Aut(C)$, and a line bundle $\cL$, we denote this integer by $\chi_{\cL}(C, \eta)$. 
If $\cL=\lambda_{m}$, we write simply $\chi_{m}(C,\eta)$. Furthermore, if $\Aut(C) \simeq \mathbb{G}_m$, then we write $\chi_{\cL}(C)$ or $\chi_{m}(C)$, where the one-parameter subgroup $\eta\co  \GG_m \rightarrow \Aut(C)$ is understood to be the identity.\footnote{Note that, in general, the integers $\chi_{m}(C)$ are only defined up to sign since the choice of isomorphism $\Aut(C) \simeq \mathbb{G}_m$ depends on a sign. The ratios  $\chi_{l}(C)/\chi_{m}(C)$ however are well-defined, and this is all we will need.}

\begin{lem}\label{L:Observations}
Suppose the modularity principle for the log MMP for $\M_g$ holds and that 
$\bar{\cM}_g(\alpha)$ is $S_2$ and the locus 
of worse-than-nodal curves in $\bar{\cM}_g(\alpha)$ has codimension at least $2$. Then
for $c(\alpha):=\frac{13\alpha-13}{2-\alpha}$, some multiple of the $\mathbb{Q}$-line bundle $c(\alpha)\lambda_1+\lambda_2$ descends to a $\mathbb{Q}$-line bundle on 
$\M_g(\alpha)$.
\end{lem}
\begin{proof}
Since  $\lambda_2 = 13\lambda_1-\delta$ on $\bar{\cM}_{g}$, we have
$K_{\bar{\cM}_{g}}+\alpha\delta \sim c(\alpha)\lambda_1+\lambda_2.$
Now consider the commutative diagram
\[
\xymatrix{
\bar{\cM}_{g} \ar[d] \ar@{-->}[r] &\bar{\cM}_{g}(\alpha)\ar[d]^{\phi}\\
\bar{M}_{g} \ar@{-->}[r]^{f}&\bar{M}_g(\alpha)\\
}
\]
By the definition of $\M_{g}(\alpha)$, the divisor class $K_{\bar{\cM}_{g}}+\alpha\delta$ pushes forward to a $\QQ$-Cartier divisor class on $\M_{g}(\alpha)$. We claim that $\phi^* f_*(K_{\bar{\cM}_{g}}+\alpha\delta)=c(\alpha)\lambda_1+\lambda_2$ on $\bar{\cM}_g(\alpha)$. Evidently, this equality holds on $\bar{\cM}_g(\alpha) \cap \bar{\cM}_{g}$. Since the complement of $\bar{\cM}_g(\alpha) \cap \bar{\cM}_{g}$ has codimension $2$ in $\bar{\cM}_g(\alpha)$ and $\bar{\cM}_g(\alpha)$ is $S_2$, equality holds over all of $\bar{\cM}_g(\alpha)$. 

\end{proof}

As a consequence, we obtain:
\begin{prop}\label{C:CriticalAlpha}
Suppose the modularity principle for the log MMP for $\M_g$ holds for $\alpha$ 
such that $\bar{\cM}_g(\alpha)$ is $S_2$ and the locus 
of worse-than-nodal curves in $\bar{\cM}_g(\alpha)$ has codimension at least $2$.
Let $C$ be a curve in $\bar{\cM}_{g}(\alpha)$.  Let $\eta\co  \mathbb{G}_m \rightarrow \Aut(C)$ be any 
one-parameter subgroup. Then either $\chi_m(C,\eta)=0$ for all $m$ or
$$
\alpha=\frac{13-2\left(\frac{\chi_2(C,\eta)}{\chi_1(C,\eta)} \right)}{13-\left(\frac{\chi_2(C,\eta)}{\chi_1(C,\eta)}\right)} \ .
$$
In other words, either $\Aut(C)^\circ$ acts trivially on all the vector spaces $\bigwedge H^0(C, \omega_C^m)$, or else $\alpha$ is uniquely determined by the characters $\chi_1(C,\eta)$ and $\chi_{2}(C,\eta)$.
\end{prop}
\begin{proof} Set $c(\alpha)=\frac{13\alpha-13}{2-\alpha}$.
By Lemma \ref{L:Observations}, the line bundle $c(\alpha)\lambda_1+\lambda_2$ must descend from $\bar{\cM}_{g}(\alpha)$ to $\bar{M}_g(\alpha)$. In particular, the one-parameter subgroup $\eta\co  \mathbb{G}_m \rightarrow \Aut(C)$ must act trivially on the fiber $(c(\alpha)\lambda_1+\lambda_2)|_{[C]}$. Thus
 the character for this action, given by $c(\alpha)\chi_1(C,\eta)+\chi_2(C,\eta)$, must be $0$. 
 We conclude that either $\chi_1(C,\eta)=\chi_2(C,\eta)=0$, or
$$
c(\alpha)=-\chi_2(C,\eta)/\chi_{1}(C,\eta),
$$
as desired.
\end{proof}

Evidently, Proposition \ref{C:CriticalAlpha} says nothing whatsoever concerning curves with finite automorphisms. 
At critical values of $\alpha$ however, the stacks $\bar\cM_{g}(\alpha)$ typically contain curves 
admitting a $\GG_m$-action. 
We define the \emph{$\alpha$-value} of a complete curve $C$ with a $\mathbb{G}_m$-action $\eta\co \mathbb{G}_m \rightarrow \Aut(C)$ as
$$
\alpha(C, \eta):= \frac{13-2\left(\frac{\chi_2(C,\eta)}{\chi_1(C,\eta)} \right)}{13-\left(\frac{\chi_2(C,\eta)}{\chi_1(C,\eta)}\right)} \ . 
$$
We note that $\alpha(C, \eta) = 2 - 13 \chi_{\lambda}(C, \eta) / \chi_{\delta}(C, \eta)$ as long as the deformation space of $C$ is $S_2$ and the locus of worse-than-nodal curves has codimension at least 2.  

Proposition \ref{C:CriticalAlpha} implies that the $\alpha$-value of any complete curve 
with $\GG_m$-action
is the {\em only} $\alpha$ at which 
the curve can show up in $\Mg{g}(\alpha)$. 
Note also that whenever $\Mg{g}(\alpha)$ is constructed as a GIT quotient, the necessary condition for 
$[C]$ to be semistable is that the character of $K_{\Mg{g}(\alpha)}+\alpha\delta$ is $0$,
as this character computes the 
Hilbert-Mumford index of $[C]$ with respect to $\eta$. We discuss a connection of the character
theory with GIT in Section \ref{section-git}.



Next, we explain how to define and extract critical $\alpha$-values for an arbitrary curve singularity with a
$\GG_m$-action. Given a curve singularity $\hat{\cO}_{C,p}$ with $n$ branches and $\delta$-invariant $\delta(p)$, we 
may consider a curve of the form
$$
C=E_1 \cup \ldots \cup E_n \cup C_0,
$$
where $C_0$ is any smooth curve of genus $g-\delta(p)-n+1$ and $E_1, \ldots, E_n$ are rational curves attached to $C_0$ nodally and meeting in a singularity analytically isomorphic to $\hat{\cO}_{C,p}$ (see Figure \ref{F:J10}).

\begin{figure}[hbt]
\begin{centering}
\begin{tikzpicture}[scale=0.55]
		\node [style=black] (0) at (-12, 3.5) {};
		\node [style=black] (e1) at (-12, 4) {$E_1$};
		\node [style=black] (1) at (-11, 3.5) {};
		\node [style=black] (e2) at (-11, 4) {$E_2$};
		\node [style=black] (2) at (-10, 3.5) {};
		\node [style=black] (e3) at (-10, 4) {$E_3$};
		\node [style=black] (3) at (-8, 3.5) {};
		\node [style=black] (4) at (-7, 3.5) {};
		\node [style=black] (5) at (-6, 3.5) {};
		\node [style=black] (6) at (-8.5, 3) {};
		\node [style=black] (7) at (-6, 2.5) {};
		\node [style=black] (8) at (-3.5, 2) {};
		\node [style=black] (c0) at (-3, 2) {$C_0$};
		\node [style=black] (a) at (-7.5, 1.5) [label=left:{$y^3=x^6$}]{};
		\draw [very thick,bend right=60] (1.center) to (4.center);
		\draw [very thick,bend right=90, looseness=1.75] (2.center) to (3.center);
		\draw [very thick,bend right=15] (7.center) to (8.center);
		\draw [very thick,bend right=35] (0.center) to (5.center);
		\draw [very thick,bend left=15] (6.center) to (7.center);
\end{tikzpicture}
\end{centering}
\vspace{-0.5pc}
\caption{Rational curves $E_1, E_2$, and  $E_3$ meet in the monomial $y^3=x^6$ singularity.}\label{F:J10}
\end{figure}
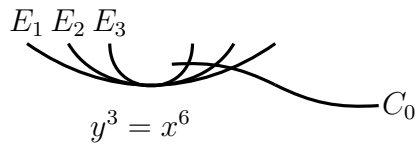

\vspace{-0.5pc}

If $\GG_m$ acts algebraically on $\hat{\cO}_{C,p}$ via $\eta$, then this 
 action extends canonically to $C$, which induces a one-parameter subgroup $\tilde \eta\co  \GG_m \rightarrow \Aut(C)$. 
The characters $\chi_1(C,\tilde \eta)$ and 
$\chi_2(C,\tilde \eta)$ depend only on the singularity $\hat{\cO}_{C,p}$ and the $\GG_{m}$-action.  
\begin{definition}
We define the \emph{$\alpha$-value of $\hat{\oh}_{C,p}$ with respect to $\eta$}, denoted by $\alpha(\hat{\cO}_{C,p}, \eta)$, as the corresponding $\alpha$-value, $\alpha(C,\tilde \eta)$, of the complete curve $C$.
\end{definition}


\begin{table}[tbh]
{\footnotesize
\renewcommand{\arraystretch}{1.7}

\begin{tabular}{| c || c | c | c  || c | c | }

\hline 
\text{Singularity type}	& $\lambda$ & $\lambda_2$ & $\delta$ &  $\alpha$-value & slope  \\
\hline
\hline
$A_{2k}: y^2 - x^{2k+1}$ & $k^2$ & $5k^2-4k+1$&  $8k^2+4k-1$ & $\frac{3k^2+8k-2}{8k^2+4k-1}$ & $\frac{8k^2+4k-1}{k^2}$ \\
\hline
$A_{2k+1}: y^2 - x^{2k+2}$ & $\frac{k^2+k}{2}$&$\frac{5k^2+k}{2}$& $4k^2+6k$& $\frac{3k+11}{8k+12}$ & $\frac{8k+12}{k+1}$\\ 
\hline
$D_{2k+1}: x(y^2-x^{2k-1})$	&$k^2$	&$5k^2-2k$	&$8k^2+2k$	& $\frac{3k+4}{8k+2}$ 	& $\frac{8k+2}{k}$ \\	
\hline
$D_{2k+2}: x(y^2-x^{2k})$	&$\frac{k^2+k}{2}$	&$\frac{5k^2+3k}{2}$	&	$4k^2+5k$ & $\frac{3k+7}{8k+10}$ & $\frac{8k+10}{k+1}$\\
\hline
$E_6: y^3-x^4$			& $8$	& $33$	& $71$ &  $38 / 71$ & $71/8$ \\
\hline
$E_7: y(y^2-x^3)$		& $7$ 	& $31$ 	& $60$ & $29/60$ & $60/7$ \\
\hline
$E_8: y^3 - x^5	$		& $14$	& $63$	& $119$ & $8/17$ & $17/2$ \\
\hline
$y^3-x^6$				& $7$ & 34 & $57$ & $23/57$ & $57/7$ \\
\hline
$y^3-x^7$				& $31$	& $152$	& $251$ 	& $99/251$ & $251/31 $ \\
\hline
$y^3 -x^8$			& $42$	& $211$	& $335$ & $124/335$ & $335/42 $ \\
\hline
$T_{p,q}: y^p-x^q $& \multicolumn{5}{|c|}{\SMALL{See  Proposition \ref{P:toric-family} for character values}}	\\
\hline
\SMALL{monomial unibranch}& \multirow{2}{*}{$\sum b_i$} &  \multirow{2}{*}{$(2k-1)^2 + \sum b_i$}
& \multirow{2}{*}{\SMALL{$12 \sum b_i - (2k-1)^2$}}  & \multirow{2}{*}{\tiny{$\dfrac{11 \sum b_i - 2 (2k-1)^2}{12 \sum b_i - (2k-1)^2}$}}
& \multirow{2}{*}{$12  - \frac{ (2k-1)^2}{\sum b_i}$} \\
\SMALL{with gaps $\{b_1, \dots, b_k\}$} & & & & & \\
\hline
\SMALL{Ribbon $C_{\ell}$}	& $g\left(\ell-\frac{g-1}{2}\right)$ & \SMALL{$(5g-4)(\ell-\frac{g-1}{2})$} & \SMALL{$(8g+4)(\ell-\frac{g-1}{2})$} & $\frac{3g+8}{8g+4}$ & $8+\frac{4}{g}$ \\
\hline
\end{tabular} }
\hfill
\caption{Character values of Gorenstein singular curves from Section \ref{S:character-computations}}
\label{T:table-characters}
\end{table}
\vspace{-1pc}

We compute the character values of all ADE, toric, and monomial unibranch Gorenstein singularities, as well as of ribbons, in Section \ref{S:character-computations}. The results are displayed in Table \ref{T:table-characters}.
We expect that the $\alpha$-values displayed in the table are the $\alpha$-values at which the given
singularity type first appears on a curve in $\bar{\cM}_{g}(\alpha)$. There are two caveats. First, the $\alpha$-value 
depends not only on the analytic isomorphism type of a singularity but also on the global geometry of the curve. 
This dependence is described in Section \ref{S:dangling} below. Second, 
there is no guarantee that at the prescribed $\alpha$-values exactly the predicted singularities appear. 
Theorem \ref{theorem-character-intersection} below is the first step towards confirming these predictions.  
In addition, we note that our predictions for when $A$ singularities arise agree with the computations of Hyeon, who uses different heuristics \cite{hyeon_predictions}.

\subsection{Dangling singularities}\label{S:dangling} \label{section-dangling}
We now explain a variant of the above ideas applied to curves with {\em dangling} singularities; such 
singular curves have not yet appeared in the work of Hassett and Hyeon but we expect them to play an important role in the future stages of the program. 
We define a collection of modified $\alpha$-values associated to a multi-branch singularity. If $\hat{\cO}_{C,p}$ is any curve singularity with $n \geq 2$ branches, we may enumerate the branches, and for any non-empty subset $S \subset \{1, \ldots, n\}$, we consider a curve of the form
$$
C^S=E_1 \cup \ldots \cup E_n \cup C_0,
$$
where each $E_i$ is $\PP^1$ with two distinguished points $0$ and $\infty$, 
each $E_i$ with $i \in S$ meets $C_0$ nodally at $\infty$, and all the $E_i$'s are glued along singularity of type 
$\hat{\cO}_{C,p}$ at $0$ (see Figure \ref{F:dangling}). 
\begin{figure}[hbt]
\begin{centering}
\begin{tikzpicture}[scale=0.38]
		\node [style=black] (n1) at (-0.5, 3) [label=left:$E_1$]{};
		\node [style= black] (1) at (-3, 2) [label=left:$E_3$]{};
		\node [style= black] (2) at (2, 2) {};
		\node [style= black] (3) at (4, 0) [label=right:$C_0$] {};
		\node [style= black] (4) at (1, -1.25) {};
		\node [style= black] (5) at (-3, -2) [label=left:$E_2$]{};
		\node [style= black] (6) at (2, -2) {};
		\node [style= black] (7) at (-2, -2.5) {};
		\node [style= black] (n2) at (-0.5, -3) {};
		\draw [very thick, bend right=75, looseness=1.41] (1.center) to (2.center); 
		\draw [very thick, bend left=75, looseness=1.41] (5.center) to (6.center); 
		\draw  [very thick] (n2.center) to (n1.center); 
		
		\draw [very thick, bend right=45] (7.center) to (4.center); 
		\draw [very thick, bend left=45] (4.center) to (3.center); 
\end{tikzpicture}
\end{centering}
\vspace{-1pc}
\caption{Dangling $D_{6}^{\{1,2\}}$-singularity.}\label{F:dangling}
\end{figure}
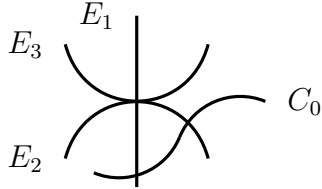
\vspace{-0.2pc}

As before, if $\GG_m$ acts algebraically on $\hat{\cO}_{C,p}$ via $\eta$, there is an induced one-parameter subgroup $\tilde \eta\co  \GG_m \rightarrow \Aut(C)$.  We define the \emph{$\alpha$-value of $\hat{\oh}_{C,p}$ with respect to $S$ and $\eta$}, denoted by $\alpha^{S}(\hat{\cO}_{C,p}, \eta)$, as the corresponding $\alpha$-value, $\alpha(C^S,\tilde \eta)$, of the complete curve $C^S$.
In this notation, $\alpha^{[n]}(\hat{\cO}_{C,p})$ is the standard $\alpha$-invariant defined above. In general, the invariants $\alpha^{S}(\hat{\cO}_{C,p})$ will depend on the subset $S$, which reflects the fact that curves $C^S$ may appear in the moduli stack $\bar{\cM}_{g}(\alpha)$ at different values of $\alpha$. 
\begin{figure}[hbt]
\begin{centering}
\begin{tikzpicture}[scale=0.4]
		\node [style=black] (0) at (-8, 4)  {};
		\node [style=black] (y) at (-7.6, 4.2) [label=left:{\SMALL $C_0$}] {};
		\node [style=black] (1) at (-9, 2) {};
		\node [style=black] (x) at (-8.7, 2) [label=left:{\small ${g=2}$}] {};
		\node [style=black] (2) at (-6, 1.5) {};
		\node [style=black] (z) at (-6.5, 1.7) [label=right:$\PP^1$] {};
		\node [style=black] (3) at (-9, 1) {};
		\node [style=black] (4) at (-8, 1) {};
		\node [style=black] (5) at (-10, 0) {};
		\node [style=black] (6) at (-8, 0) {};
		\node [style=black] (7) at (-7, 0) {};
		\node [style=black] (8) at (-6, 0) {};
		\node [style=black] (9) at (-12.5, -1) {};
		\node [style=black] (d) at (-12.1, -1)  [label=left:{\SMALL $C_0$}] {};
		\node [style=black] (10) at (-11, -1) {};
		\node [style=black] (11) at (-5, -1) {};
		\node [style=black] (12) at (-4, -1) {};
		\node [style=black] (w) at (-2.1, -1) {};
		\node [style=black] (r) at (-1.9, -1) [label=left:{\SMALL $C_0$}] {};
		\node [style=black] (13) at (-13.5, -3) {};
		\node [style=black] (q) at (-13.1, -3) [label=left:{\small ${g=2}$}] {};
		\node [style=black] (14) at (-13.5, -4) {};
		\node [style=black] (15) at (-12.5, -4) {};
		\node [style=black] (16) at (-5, -4) {};
		\node [style=black] (17) at (-11.5, -5) {};
		\node [style=black] (18) at (-4, -5) {};
		\node [style=black] (eq) at (-4, -3) [label=right:{\small ${y^2=x^{6}}$}] {};
		\node [style=black] (19) at (-3, -5) {};
		\node [style=black] (e) at (-3.3, -5) [label=right:$\PP^1$] {};
		\draw [very thick,bend right=300] (12.center) to (16.center);
		\draw [very thick, out=105, looseness=5.80, in=105] (18.center) to (19.center);
		\draw [very thick,bend right=300] (0.center) to (3.center);
		\draw [very thick,bend right=45, looseness=1.5] (15.center) to (17.center);
		\draw [very thick, bend right=45, looseness=1.5] (4.center) to (7.center);
		\draw [very thick, bend left=45, looseness=1.5] (1.center) to (4.center);
		\draw [very thick, ->]  (5.center) to (10.center);
		\draw [very thick, bend left=45, looseness=1.5] (13.center) to (15.center);
		\draw [very thick, bend right=300] (9.center) to (14.center);
		\draw  [very thick] (6.center) to (2.center);
		\draw [very thick, ->] (8.center) to (11.center);
\end{tikzpicture}
\end{centering}
\caption{Given a smoothing of a curve with a genus $2$ tail attached at an arbitrary point $p$, after blowing up the conjugate point of $p$ and contracting the genus $2$ curve, one obtains a dangling $\PP^1$ attached at an oscnode.}\label{F:dangling-genus-2}
\end{figure}
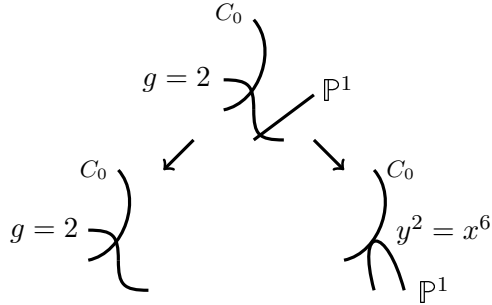

{\footnotesize
\begin{table}[htb] 
\renewcommand{\arraystretch}{1.6}
\begin{tabular}{|c || c | c | c  || c | c |}
\hline
\text{Dangling type}	& $\lambda$ & $\lambda_2$ & $\delta$ &  $\alpha$-value & slope  \\
\hline
\hline
$A_{2k}^{\{\}}: y^2 - x^{2k+1}$ & $k^2$ 	& $5k^2-4k$		& $8k^2+4k$	& $\frac{3k^2+8k}{8k^2+4k}$	& $\frac{8k+4}{k}$\\ 
\hline
$A_{2k+1}^{\{\}}: y^2 - x^{2k+2}$ & $\frac{k^2+k}{2}$ 	& $\frac{5k^2+k-4}{2}$		& $4k^2+6k+2$ & $\frac{3k^2+11k+8}{8k^2+12k+4}$ & $\frac{8k^2+12k+4}{k^2+k}$\\ 
\hline
$A_{2k+1}^{\{1\}}: y^2 - x^{2k+2}$ & $\frac{k^2+k}{2}$ 	& $\frac{5k^2+k-2}{2}$	& $4k^2+6k+1$ & $\frac{3k^2+11k+4}{8k^2+12k+2}$ & $\frac{8k^2+12k+2}{k^2+k}$\\ 
\hline
$D_{2k+1}^{\{1\}}: x(y^2-x^{2k-1})$	&$k^2$ 	& $5k^2-2k-1$	& $8k^2+2k+1$ & $\frac{3k^2+4k+2}{8k^2+2k+1}$& $\frac{8k^2+2k+1}{k^2}$\\	
\hline
$D_{2k+2}^{\{1\}}: x(y^2-x^{2k})$	& $\frac{k^2+k}{2}$ & $\frac{5k^2+3k-4}{2}$ & $4k^2+5k+2$ 	& $\frac{3k^2+7k+8}{8k^2+10k+4}$ 	& $\frac{8k^2+10k+4}{k^2+k}$\\
\hline
$D_{2k+2}^{\{1,2\}}: x(y^2-x^{2k})$	& $\frac{k^2+k}{2}$ 	&  $\frac{5k^2+3k-2}{2}$	& $4k^2+5k+1$ &$\frac{3k^2+7k+4}{8k^2+10k+2}$	& $\frac{8k^2+10k+2}{k^2+k}$  \\
\hline
$E_6^{\{\}}: y^3-x^4$			& $8$	& $32$	& $72$ &  $5 / 9$ & $9$ \\
\hline
$E_7^{\{1\}}: y(y^2-x^3)$		& $7$ 	& $30$ 	& $61$ & $31/61$ & $61/7$ \\
\hline
$E_8^{\{\}}: y^3 - x^5	$		& $14$	& $62$	& $120$ & $29/60$ & $60/7$ \\
\hline
\hline
Dangling chains (see \ref{S:dangling-chains}) & \multicolumn{2}{|c|}{$\lambda$} & $\delta$ \\
\cline{1-4}
\cline{1-4}

$A_{2i+1/2j+1}$ & \multicolumn{2}{|c|}{$\frac{j^2+j-i^2-i}{2}$}	& $4j^2+6j-4i^2-6i+1$	\\ 
\cline{1-4}
$A_{2i+1/2j}$ & \multicolumn{2}{|c|}{$j^2-(\frac{i^2+i}{2})$}		& $8j^2+4j-4i^2-6i-1$ \\
\cline{1-4}
\end{tabular}
\smallskip
\caption{Character values for dangling ADE singularities from Section \ref{S:character-computations} }
\label{T:table-dangling}
\end{table}
}
The first example of this phenomenon should occur with the oscnode $(y^2=x^6)$. As seen in Table \ref{T:table-dangling}, the $\alpha$-invariant of the oscnode is $17/28$ reflecting the fact that we expect oscnodes to replace genus $2$ bridges attached by conjugate points. By contrast, the $\alpha$-invariant of the dangling oscnode $A_{5}^{\{1\}}$ is $19/29$. The key point is that this is precisely the threshold $\alpha$-value at which $\Delta_2$ is covered by $(K+\alpha\delta)$-negative curves, and indeed one can replace arbitrary genus $2$ tails by a dangling oscnode, using the blow-up/blow-down procedure pictured in Figure \ref{F:dangling-genus-2}.

While it would be too laborious to compute the associated $\alpha^{S}$-values even for toric planar 
singularities, we will do a sample in order to given an indication. In Table \ref{T:table-dangling}, we list 
$\alpha$-values for all dangling ADE singularities.  Note that since the branches of any $A_k$ or toric singularity are isomorphic the only relevant feature of the subset $S\subset \{1, \ldots, n\}$ is the size. For $D_{2k+1}$ singularities, we use the labeling ``1'' for the smooth branch and ``2'' for the singular branch, and for $D_{2k+2}$-singularities, we use ``1'' for the smooth branch with unique tangent direction and ``2,3'' for the tangent branches. 
Similarly for the $E_7$ singularity, we use the labeling ``1'' for the smooth branch and ``2" for the singular branch.

\subsection{Chains of dangling singularities}\label{S:dangling-chains}
We also predict that
 in future steps in the log MMP for $\bar{M}_g$ it will be necessary to parameterize curves admitting certain chains of dangling singularities.  Rather than defining a general theory of chains of dangling singularities, we will introduce two particular sequences which we anticipate will arise before $\alpha = 5/9$.

 We will say that a genus $g$ curve $C$ has an {\em $A_{2i+1/2j+1}$-singularity} (resp. {\it $A_{2i+1/2j}$-singularity}) if $C$ is of the form
 $$C = C_0\cup E_1 \cup E_2 \cup E_3  \qquad (\text{resp. }\, C = C_0\cup E_1 \cup E_2 ),$$
 where $C_0$ is a genus $g-i-j$ curve, each $E_k$ is a smooth rational curve,  $E_1$ meets $C_0$ at a node, $E_2$ meets $E_1$ at an $A_{2i+1}$ singularity, and $E_3$ meets $E_2$ at an $A_{2j+1}$ singularity (resp.  $E_2$ has a monomial $A_{2j}$-singularity); see Figure \ref{Fig:dangling-slash}.
\begin{figure}[hbt]
\begin{centering}
\begin{tikzpicture}[scale=0.7]
		\node  (0) at (-4, 2.5) {};
		\node  (1) at (-3, 2.5) {};
		\node  (2) at (-2, 2.5) {};
		\node  (3) at (-1, 2.5) {};
		\node  (4) at (0, 2.5) {};
		\node  (5) at (1, 2.5) {};
		\node  (6) at (-3.75, 1.75) {};
		\node  (7) at (-3, 1.75) {};
		\node  (8) at (-0.75, 1.75) {};
		\node  (9) at (0, 1.75) {};
		\node  (10) at (-3.75, 1) {};
		\node  (11) at (-3, 1) {};
		\node  (b) at (-3.25, 1.5) [label=right:{\footnotesize ${A_{2i+1}}$}] {};
		\node  (12) at (-0.75, 1) {};
		\node  (13) at (0, 1) {}; 
		\node  (a) at (-0.25, 1.5) [label=right:{\footnotesize ${A_{2i+1}}$}] {};
		
		\node  (c) at (-0.25, 0.2) [label=right:{\footnotesize ${A_{2j}}$}] {};
		\node  (d) at (-3.25, 0.2) [label=right:{\footnotesize ${A_{2j+1}}$}] {};	
		\node  (14) at (-3.75, 0.5) {};
		\node  (15) at (-3, 0.5) {};
		\node  (16) at (-0.75, 0) {};
		\node  (17) at (0, 0) {};
		\node  (18) at (-3.75, -0.25) {};
		\node  (19) at (-3, -0.25) {};
		\node  (20) at (-0.75, -0.75) {};
		\node  (21) at (-3, -1) {};
		\node  (22) at (0, -1) {};
		\node  (23) at (0, -1.75) {};
		\draw [very thick, out=-90, looseness=1.50, in=15] (13.center) to (16.center);
		\draw [very thick, out=0, looseness=1.25, in=82] (16.center) to (22.center);
		\draw [very thick, bend right=45] (3.center) to (5.center);
		\draw [very thick, bend left=90, looseness=1.75] (18.center) to (19.center);
		\draw [very thick, bend left=90, looseness=1.75] (12.center) to (13.center);
		\draw [very thick, bend left=90, looseness=1.75] (10.center) to (11.center);
		\draw [very thick] (1.center) to (7.center);
		\draw [very thick, bend left=90, looseness=1.50] (7.center) to (6.center);
		\draw [very thick, bend right=45] (0.center) to (2.center);
		\draw [very thick] (11.center) to (15.center);
		\draw [very thick, bend right=90, looseness=1.50] (14.center) to (15.center);
		\draw [very thick] (4.center) to (9.center);
		\draw [very thick] (19.center) to (21.center);
		\draw [very thick, bend left=90, looseness=1.50] (9.center) to (8.center);
\end{tikzpicture}
\end{centering}
\vspace{-2pc}
\caption{ $A_{2i+1/2j+1}$ and $A_{2i+1/2j}$-singularities.}\label{Fig:dangling-slash}
\end{figure}
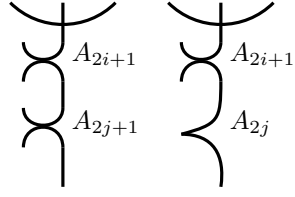 

\section{Predictions for the log MMP for $\M_g$}
\label{S:predictions}
Using heuristics provided by both intersection theory and character computations, we offer predictions 
in Table \ref{table-predictions} for modular interpretations of $\Mg{g}(\alpha)$ for $\alpha \ge 5/9$. 
(For small $g$ these predictions have to be modified. 
For example, in $\M_3$ Weierstrass genus $2$ tails are also elliptic tails and so
$\alpha=2/3$ is not a threshold value, as can be seen from \cite{hyeon-lee_genus3}.)
\addtocounter{footnote}{1}
\footnotetext[\value{footnote}]{Although it goes beyond the scope of this paper, the replacement of genus $2$
bridges requires both $D_6^{\{1,2\}}$ singularities {\em and} certain non-reduced double line bridges;
in the case of genus $5$ this is worked
out explicitly in \cite{fedorchuk-genus5}.}
{\footnotesize
\begin{table}[htb]
\renewcommand{\arraystretch}{1.3}
\begin{tabular}{l || c | c}
\multirow{2}{*}{$\alpha$-value}    & Singularity type    & \multirow{2}{*}{Locus removed at
$\alpha - \epsilon$}   \\
                &  added at $\alpha$\\
\hline
\hline
$9/11$            & $A_2$                & \small{elliptic tails attached nodally}\\
\hline
$7/10$            & $A_3$                & \small{elliptic bridges attached nodally } \\
\hline
$2/3$            & $A_4$                & \small{genus $2$ tails attached nodally at a Weierstrass point}  \\
\hline
\multirow{2}{*}{$19/29$}
        & $A_5^{\{1\}}$            & \small{genus $2$ tails attached nodally}\\
        & $A_{3/4}$                & \small{genus $2$ tails attached tacnodally at a Weierstrass point}\\
\hline
$12/19$
        & $A_{3/5}$                & \small{genus $2$ tails attached tacnodally}\\       
       
\hline
$17/28$        & $A_5$                & \small{genus $2$ bridges attached nodally at
conjugate points}         \\
\hline
$49/83$        & $A_6$                & \small{hyperelliptic genus $3$ tails attached nodally at
a Weierstrass point}            \\
\hline
$32/55$        & $A_7^{\{1\}}$            & \small{hyperelliptic genus $3$ tails attached
nodally}        \\
\hline
$42/73$        & $A_{3/6}$            & \small{hyperelliptic genus $3$ tails attached
tacnodally at a Weierstrass point}        \\
\hline
\multirow{7}{*}{$5/9$}                    & $D_4$    &    \small{elliptic triboroughs attached nodally}   \\
        &$D_5$    & \small{genus $2$ bridges attached nodally at a
Weierstrass and free point}\\

& $D_6^{\{1,2\}}$     & \small{genus $2$ bridges attached
nodally at two free points$^{\decimal{footnote}}$ } \\
       
& $A_{3/7}$                & \small{hyperelliptic genus $3$ tails attached tacnodally}\\
    & $A_{5/7}$                & \small{hyperelliptic genus $3$ tails attached oscnodally}\\
        & $A_{5/6}$                & \small{hyperelliptic genus $3$ tails attached oscnodally at a Weierstrass point}\\
       
& $A_7$                & \small{hyperelliptic genus $3$ bridges attached
nodally at conjugate points} \\
\end{tabular}
\medskip
\caption{Predictions for the log MMP for $\alpha \ge 5/9$}
\label{table-predictions}
\end{table}
}
\begin{remark}\label{R:chamber-remark}
There are in fact other singular curves with $\GG_m$-action whose characters allow them to appear
in $\Mg{g}(\alpha)$ for $\alpha\in [5/9, 1]$. However, they can be excluded by geometric considerations.
Heuristically, a necessary condition for a curve to appear in $\Mg{g}(\alpha)$ is that it is an isotrivial specialization
of curves in $\Mg{g}(\alpha+\epsilon)$. For example, an $A_{5/4}$-curve has $\alpha$-value $9/11$ but obviously
is not an isotrivial specialization of any stable curve.
\end{remark}

Giving a complete description of $\bar{\cM}_g(\alpha)$ is much more subtle than generally describing the singularities added and the loci removed.  For instance, after removing elliptic tails (connected genus $1$ subcurves which are attached nodally) at $\alpha = 9/11 - \epsilon$, in each subsequent moduli stack parameterizing curves with additional singularities, one needs to redefine what is meant by an elliptic tail by specifying the allowed attaching singularities.

\section{Character theory computations}
\label{S:character-computations}
\subsection{Computing the characters of $\lambda$ and $\lambda_2$}
Suppose we are given a curve $C$ and a one-parameter subgroup $\eta\co \GG_m \to \Aut(C)$. Then there is an induced $\GG_m$-action on the sequence of one-dimensional vector spaces
$$
\lambda_{m}|_{[C]}:=\bigwedge \HH^0(C, \omega_{C}^m).
$$
For many classes of singularities, the induced character $\chi_{m}(C, \eta) \in \mathbb{Z}$ of this action can be explicitly computed.  In this section, we will calculate these characters for $A_{2k}$, $D_{2k+2}$ singularities, elliptic $m$-fold points, monomial unibranch Gorenstein singularities, and ribbons.  The same procedure can be used to compute these characters for all ADE and toric singularities, and these results are listed in Table \ref{T:table-characters}. 

Throughout this section, we will use the following basic result about the dualizing sheaf $\omega_C$ 
of a reduced singular curve $C$.
Let $\nu\co \tilde{C} \rightarrow C$ be the normalization of $C$ and consider the sheaf $\Omega_{\tilde{C}} \otimes 
K(\tilde{C})$ of rational differentials on $\tilde{C}$. 
Then $\omega_C\subset \nu_*\bigl(\Omega_{\tilde{C}} \otimes K(\tilde{C})\bigr)$ is the subsheaf of {\em Rosenlicht differentials} 
defined as follows: A differential $\omega\in \nu_* \bigl(\Omega_{\tilde{C}} \otimes K(\tilde{C}) \bigr)$ is Rosenlicht at 
$p\in C$ 
if for every function $f \in \cO_{C,p}$
$$
\sum_{p_i \in \nu^{-1}(p)}\text{Res}_{\, p_i} \, (f \,\omega)=0.
$$
See \cite[Prop.6.2]{barth} for the proof of this fact in the 
analytic setting or \cite[Ch.IV]{serre-corps} for a general discussion of duality on singular curves.

\begin{example}[$A_{2k}: y^2=x^{2k+1}$]
Let $C=C_{0} \cup E$, where $E$ is a smooth rational curve with a higher cusp $y^2=x^{2k+1}$ at zero, 
attached nodally at infinity to $p\in C_{0}$. 
If $t$ is a uniformizer at zero, then $dt/t^{2k}$ is a generator for 
$\omega_{C}$ at the cusp, and we may write down a basis for $\HH^0(C, \omega_{C})$ as follows:
$$
\left(0, \frac{dt}{t^{2k}}\right),\left(0, \frac{dt}{t^{2k-2}}\right), \ldots, \left(0, \frac{dt}{t^2}\right), (\omega_{1}, 0), \ldots, \left(\omega_{g-k},0 \right),
$$
where $\omega_{1}, \ldots, \omega_{g-k}$ is a basis for $\omega_{C_0}$.
This basis diagonalizes the $\GG_m$-action 
$\eta\co t \mapsto \lambda^{-1} t$ with weights $(2k-1), (2k-3), \ldots, 1$. 
Thus, the character $\chi_{1}(C, \eta)$ is
$$
\chi_{1}=\sum_{i=1}^{k}(2i-1)=k^2.
$$
Similarly, we may write down a basis for $\HH^0(C, \omega_{C}^2)$ as
\begin{multline*}
\left(0, \frac{(dt)^2}{t^{4k}}\right),\left(0, \frac{(dt)^2}{t^{4k-2}}\right), \ldots, \left(0, \frac{(dt)^2}{t^{2k}}
\right), \left(0, \frac{(dt)^2}{t^{2k-1}}\right), \ldots, \left(w_{0}, \frac{(dt)^2}{t^{2}}\right),
\\ (w_{1}, 0), \ldots, \left(w_{3g-3k-2},0 \right),
\end{multline*}
where $w_{1}, \ldots, w_{3g-3k-2}$ is basis for $\HH^0(C_0,\omega_{C_0}^2(p))$, and 
$w_{0}$ is an appropriately chosen element of $\HH^0(C_0,\omega_{C_0}^2(2p)) \setminus \HH^0(C_0,\omega_{C_0}^2(p))$.

Thus, the character $\chi_{2}(C,\eta)$ is given by
$$\chi_{2}=\sum_{i=0}^{k-1}(2k+2i)+\sum_{i=0}^{2k-2}i =5k^2-4k+1.$$

\end{example}

\begin{example}[$D_{2k+2}: x(y^2-x^{2k})=0$]
Let $C=C_0\cup E$, where $C_0$ is a genus $g-k-2$ curve and $E=E_1\cup E_2\cup E_3$ is the union 
of three rational curves at the monomial $D_{2k+2}$ singularity.
The normalization map is given by:
\begin{align*}
\left(
\begin{matrix}
x \\ y
\end{matrix}
\right)&\rightarrow
\left(
\begin{matrix}
0 & \ \ t_2 & \ \ t_3\\
t_1& -t_2^{k} & \ \ t_3^{k}
\end{matrix}
\right). 
\end{align*}
A local generator for $\omega_C$ is $\omega_0=\left(\dfrac{2dt_1}{t_1^2}, \dfrac{dt_2}{t_2^{k+1}},  -\dfrac{dt_3}{t_3^{k+1}}\right)$.
 If $\omega_1,\dots,\omega_{g-k-2}$ is a basis of $\HH^0(C_0,\omega_{C_0})$ and 
 $v_1\in \HH^0(C_0,\omega_{C_0}(p_1+p_2))$, $v_2\in \HH^0(C_0,\omega_{C_0}(p_1+p_2+p_3))$
 are appropriately chosen differentials, then the basis
\begin{align*}
 (\omega_1,0), \dots, (\omega_{g-k-2},0), (0,\omega_0), (0, x\omega_0), 
\dots, (0, x^{k-1}\omega_0), (v_1, x^k\omega_0), (v_2, y\omega_0),
\end{align*} 
of $\HH^0(C,\omega_C)$ diagonalizes the action $(t_1,t_2,t_3) \mapsto (\lambda^{-k}t_1,\lambda^{-1} t_2, \lambda^{-1} t_3)$. Thus,
$$
\chi_{1}=1+2+\cdots+k=\frac{k(k+1)}{2}.
$$
A generator for $\omega_C^2$ is $\omega_0^2=\left(\dfrac{4(dt_1)^2}{t_1^4}, \dfrac{(dt_2)^2}{t_2^{2k+2}}, \dfrac{(dt_3)^2}{t_3^{2k+2}} \right)$, so we write out an array of $(3k+3)$ quadratic differentials with non-zero weight ($2k+1$ in 
the first column, $k+1$ in the second column, $1$ in the third):
{\large
\begin{align*}
\begin{matrix}
(\frac{4(dt_1)^2}{t_1^4}, \frac{(dt_2)^2}{t_2^{2k+2}}, \frac{(dt_3)^2}{t_3^{2k+2}})&(\frac{4(dt_1)^2}{t_1^3}, \frac{(dt_2)^2}{t_2^{k+2}}, \frac{-(dt_3)^2}{t_3^{k+2}})&(\frac{4(dt_1)^2}{t_1^2}, \frac{(dt_2)^2}{t_2^{2}}, \frac{(dt_3)^2}{t_3^{2}})\\
& & \\
(0, \frac{(dt_2)^2}{t_2^{2k+1}}, \frac{(dt_3)^2}{t_3^{2k+1}})&(0, \frac{(dt_2)^2}{t_2^{k+1}}, \frac{-(dt_3)^2}{t_3^{k+1}})&\\
\vdots&\vdots&\\
\end{matrix}
\end{align*}
}
By summing the weights, we find:
$$
\chi_2 = \sum_{i=0}^{2k}i+\sum_{i=0}^{k}i=(2k+1)(2k)/2+(k+1)k/2=\frac{5k^2+3k}{2}.
$$
\end{example}


\begin{example}[Elliptic $m$-fold points] Let $m\geq 3$. 
An elliptic $m$-fold point $E$ is a Gorenstein union of $m$ general 
lines through a point in $\PP^{m-1}$ \cite{smyth-compositio}. Every such singularity is isomorphic to 
the cone over points
 $p_1=(1,0,\dots, 0)$, $p_2=(0,1,\dots,0)$, $\dots$, 
$p_{m-1}=(0, 0, \dots, 1)$, and $p_m=(1,\dots, 1)$, with the vertex at $0\in \AA^{m-1}$. 
If $(x_1,\dots, x_{m-1})$ are coordinates centered at the vertex then the normalization map from $m$ copies
of $\PP^1$ to $E$ 
is given by 
\begin{align*}
\left(
\begin{matrix}
x_1\\
\vdots\\
\vdots\\
x_{m-1}
\end{matrix}
\right)&\rightarrow
\left(
\begin{matrix}
t_1& 0& \hdots  & 0 & t_{m}\\
0&t_2& \ddots & \vdots & t_{m} \\
\vdots& \ddots& \ddots& 0& \vdots  \\
0 & \hdots &0 & t_{m-1} & t_{m}
\end{matrix}
\right).
\end{align*}
We let $C$ be the singular curve obtained by attaching $E$ to a smooth curve $C_0$ 
nodally at points $p_1,\dots, p_m$. 
A generator for $\omega_C$ in the neighborhood of the $m$-fold point is 
$$
\omega_0
=\left(\frac{dt_1}{t_1^2}, \ldots, \frac{dt_{m-1}}{t_{m-1}^2}, -\frac{dt_m}{t_{m}^{2}}\right).
$$
In fact,  $\omega_0$ 
spans the only weight space of $\HH^0(C,\omega_C)$ with a non-zero weight.
Thus, $\chi_{1}(C)=1$. 
A generator for $\omega_{C}^2$ in the neighborhood of the $m$-fold point is 
$$
\omega_0^2=\left(\frac{(dt_1)^2}{t_1^4}, \ldots, \frac{(dt_{m-1})^2}{t_{m-1}^4},  \frac{(dt_m)^2}{t_{m}^{4}}\right),
$$
and the only weight spaces of $\HH^0(C,\omega_C^2)$ with non-zero weights are 
spanned by 
$$(\omega_0^2, 0), (x_1\omega_0^2, 0), \ldots, (x_{m-1}\omega_0^2, 0).$$
It follows that $\chi_{2}(C)=2+(m-1)=m+1$.
Thus, the $\alpha$-invariant of the elliptic $m$-fold point is
$$\alpha=\frac{11-2m}{12-m}.$$
\end{example}

\begin{example}[Monomial unibranch singularities] \label{E:monomial-unibranch} 
Let $C$ be the projective closure of the curve
$\spec k[t^{n}: n \geq 0, n\notin \{b_1,\dots,b_k\}]$. Clearly, $p_a(C)=k$,
the normalization of $C$ is $\PP^1$, and $C$ has an isolated monomial singularity at $t=0$. 
From now on we assume that $C$ is Gorenstein.
The condition for $C$ to be Gorenstein is that the gap sequence $\{b_1=1, \dots, b_k\}$ 
is symmetric: 
$$
n\in \{b_1,\dots,b_k\} \Leftrightarrow 2k-1-n\notin \{b_1,\dots,b_k\}.
$$ 
In particular, $b_k= 2k-1$. Evidently, a generator for $\omega_C$ in a neighborhood of zero is given by 
$dt/t^{b_k+1}$.  Therefore, we can write down bases
\begin{align*}
\HH^0(C, \omega_C)  &= \left\langle \frac{dt}{t^{b_1+1}}, \frac{dt}{t^{b_2+1}}, \ldots, \frac{dt}{t^{b_k+1}} \right\rangle \\
\HH^0(C, \omega_C^{2})  &= \left\langle \frac{(dt)^2}{t^{2b_k+2-j}} \quad : \quad j \in \{ 0, \ldots, 2b_k-2 \}\smallsetminus \{b_1,\dots,b_k\}\right\rangle
\end{align*}
and we compute
\begin{align}
\chi_1 	& = \sum_{i=1}^k b_i \ , \label{unibranch-1} \\
\chi_2	& = \sum_{j=0}^{2b_k-2} (2b_k-j) - \sum_{i=1}^k (2b_k-b_i)  
		 =  (2k-1)^2 + \sum_{i=1}^k b_i - 1. \label{unibranch-2}
\end{align}

In the case when $C' = C \cup E$ is the nodal union of $C$ and a genus $g-k$ curve attached at the infinity, 
the corresponding characters of $C'$ are $\chi_1 = \sum_{i=1}^k b_i$ and $\chi_2 = (2k-1)^2 + \sum_{i=1}^k b_i$.

For the toric singularity $x^p = y^q$ with $p$ and $q$ coprime, the local ring of the singularity is 
 $k[t^{pi+qj}: i,j\in \ZZ_{\geq 0}]$. The gap sequence $\{b_1,\dots, b_k\}$ 
 is the set of positive integers that cannot be expressed
as $pi+qj$ with $i,j\geq 0$. The study of this sequence, e.g. finding its cardinality and the largest element
is classically known in elementary number theory as the {\em Frobenius problem} \cite{frobenius-problem}.
It is well-known that the largest gap is $b_k=pq-p-q$. It is also 
easy to see that the gap sequence is symmetric: $n$ is a gap if and only if $pq-p-q-n$ is not a gap.
It follows that the genus of the singularity $x^p=y^q$ is $g=(p-1)(q-1)/2$. By \cite{frobenius-1} (see also \cite{frobenius-2}), the sum of gaps is 
\begin{equation}\label{E:sum-of-gaps}
\sum_{n=1}^{g} b_i=(p-1)(q-1)(2pq-p-q-1)/12.
\end{equation}
It follows from Equations \eqref{unibranch-1}-\eqref{unibranch-2} that
\begin{align*}
\chi_1&=\frac{1}{12}(p-1)(q-1)(2pq-p-q-1), \\ 
\chi_2 &= (pq-p-q)^2+\frac{1}{12}(p-1)(q-1)(2pq-p-q-1)-1.
\end{align*}
Remarkably, intersection theory calculations of Proposition \ref{P:toric-family} give an independent 
algebro-geometric proof of the highly nontrivial combinatorial Formula \eqref{E:sum-of-gaps};
see Section \ref{S:intersection-theory} below. 
\end{example}

\begin{example}[Non-reduced curves: A case study of ribbons]
\label{E:ribbons}

The character theory is particularly suited to the study of 
non-reduced Gorenstein 
schemes. Here, we treat the case of {\em ribbons}.  Ribbons occur as certain flat limits (in the Hilbert scheme) of canonically embedded smooth curves degenerating to hyperelliptic curves \cite{fong}.  Our exposition is self-contained but we refer the reader to \cite{BE} for a more systematic study of ribbons.

A ribbon 
is a scheme obtained by gluing together two copies of 
$\AA^1[\epsilon]:=\spec k[x,\eps]/(\eps^2)$. Precisely, let $U_1=\spec k[x,\eps]/(\eps^2)$
and $U_2=\spec k[y,\eta]/(\eta^2)$, and let $(U_1)_x$ and $(U_2)_y$ be the corresponding open affine
subschemes.
Then by \cite[p. 733]{BE} a ribbon of genus $g$ 
is given by a gluing isomorphism
$\varphi\co (U_1)_y \ra (U_2)_x$ defined by 
\begin{align*}
x &\mapsto y^{-1}-y^{-2}f(y)\eta, \\
\eps &\mapsto y^{-g-1}\eta,
\end{align*}
where $f(y)=f_1y^{-1}+\cdots+f_{g-2}y^{-(g-2)}\in \dfrac{k[y,y^{-1}]}{k[y]+y^{-g+1}k[y^{-1}]}$.

We focus here on non-split ribbons that admit a $\GG_m$-action. There are $g-2$ such ribbons, 
each given by $f(y)=y^{-\ell}$, for $\ell\in\{1,\dots, g-2\}$. Denote the ribbon 
corresponding to $f(y)=y^{-\ell}$ by $C_{\ell}$. Then
the $\GG_m$-action on $C_\ell$ is given by $t\cdot (x,y,\eps, \eta)=(tx, t^{-1}y, t^{g-\ell}\eps, t^{-\ell-1}\eta)$.

By adjunction, the sections of $\omega_{C_\ell}$ over $U_1$ 
are identified with restrictions to $U_1$ of $2$-forms
$f(x,\eps)\frac{dx\wedge d\eps}{\eps^2}$ on $\spec k[x,\eps]$, and the 
sections of $\omega_{C_\ell}$ over $U_2$ are identified with restrictions to $U_2$ of $2$-forms
$f(y,\eta)\frac{dy\wedge d\eta}{\eta^2}$ on $\spec k[y,\eta]$. With this in mind, we can write 
down $g$ linearly independent global sections of $\omega_{C_\ell}$:
\begin{align*}
\intertext{For $k=0,\dots, g-\ell-2$, take}
\omega_k &=x^{k}\frac{dx\wedge d\eps}{\eps^2}=-(y^{g-1-k}+(g-\ell-k-1)y^{g-\ell-k-2}\eta)\frac{dy\wedge d\eta}{\eta^2}, \\
\intertext{for $k=g-\ell-1, \dots, g-1$, take}
\omega_k &=(x^{k}+(\ell+k+1-g)x^{\ell+k-1}\eps)\frac{dx\wedge d\eps}{\eps^2}=-y^{g-1-k}\frac{dy\wedge d\eta}{\eta^2}.
\end{align*}
It follows that $\{\omega_i\}_{i=0}^{g-1}$ form the basis of $\HH^0(C_\ell,\omega_{C_\ell})$.
Note that we recover the second part of \cite[Theorem 5.1]{BE}, namely 
the identification of the sections of $\HH^0(C_\ell,\omega_{C_\ell})$ 
with functions 
\[
1, y, y^2, \dots, y^{\ell}, y^{\ell+1}+\eta, y^{\ell+2}+2y\eta, \dots, y^{g-1}+(g-\ell-1)y^{g-\ell-2}\eta,
\]  
under a trivialization of $\omega_{C_\ell}$ on $U_2$. 

We now proceed with character computations. Under the $\GG_m$-action above, we have that 
\[
t\cdot \omega_k=t^{k-g+\ell+1}\omega_k.
\]
Summing up the weights of the $\GG_m$-action on the basis $\{\omega_i\}_{i=0}^{g-1}$, we 
obtain the character of $\lambda$:
\[
\chi_{1}(C_\ell)=\sum_{k=0}^{g-1}(k-g+\ell+1)=g(\ell+1-g)+g(g-1)/2=g\left(\ell-\frac{g-1}{2}\right).
\]
It remains to compute the weights of the $\GG_m$-action on a basis of $\HH^0(C_{\ell},\omega_{C_{\ell}}^2)$ and 
the corresponding character $\chi_2(C_{\ell})$.
Since $h^0(C_{\ell},\omega_{C_{\ell}}^2)=3g-3$, it suffices to exhibit $3g-3$ linearly independent 
sections. One such choice is presented by 
\begin{align*} 1, y, y^2,\dots, y^{2\ell},  
&y^{2\ell+1}+y^\ell\eta, \dots, y^{\ell+g-1}+(g-\ell-1)y^{g-2}\eta, \\ 
& y^{\ell+g}
+(g-\ell)y^{g-1}, \dots, y^{2g-2}+(2g-2\ell-2)y^{2g-\ell-3}, \eta, y\eta, \dots, y^{g-3}\eta.
\end{align*}
In particular, taking into account that the weight of $\frac{dy\wedge d\eta}{\eta^2}$ is $2\ell$, 
we see that the weights of the $\GG_m$-action on $\HH^0(C_{\ell},\omega_{C_{\ell}}^2)$ are 
\begin{align*}
2\ell,2\ell-1, \dots, 2\ell-(2g-2),
& \ \ell-1, \dots, \ell-g+2.
\end{align*}
Summing up these weights, we obtain the character of $\lambda_2$:
\begin{align*}
\chi_2(C_\ell) &=2(2g-1)\ell-(g-1)(2g-1)+(g-2)\ell-(g-2)(g-1)/2 \\
&=(5g-4)(\ell-\frac{g-1}{2})=(5g-4)\chi_1(C_\ell).
\end{align*}
By Equation \eqref{relations}, the character of $\delta$ is $\chi_{\delta}(C_{\ell})=(8g+4)(\ell-\frac{g-1}{2})$. In particular, if $\ell\neq (g-1)/2$, then all three characters $\chi_1(C_\ell), \chi_2(C_\ell), \chi_\delta(C_\ell)$ are non-zero, and we have 
\[
\frac{\chi_{\delta}(C_{\ell})}{\chi_{\lambda}(C_{\ell})}=\frac{8g+4}{g}.
\]
\end{example}
\begin{remark}\label{R:anand} Generalizing the computations of Example \ref{E:ribbons} above,
Anand Deopurkar recently computed characters of Gorenstein 
$n$-ribbons\footnote{An $n$-ribbon is a non-reduced scheme supported on $\PP^1$ and
locally isomorphic to $U\times \spec k[\eps]/(\eps^n)$, where $U\subset \PP^1$ is affine.} 
with 
$\GG_m$-action and verified that always
$$
\chi_{\delta}=\frac{12(2g+n-1)n}{n^2+(4g-3)n+2-2g}\chi_{\lambda}.
$$
This recovers the ratio $\frac{8g+4}{g}$ for $2$-ribbons, gives the ratio
$\frac{36(g+1)}{5g+1}$ for $3$-ribbons (see Corollary \ref{C:stankova} and the subsequent discussion 
for the significance of this slope),
and more generally gives the same ratio $\chi_{\delta}/\chi_{\lambda}$ as that of 
the toric singularity $y^{n}=x^{2g/(n-1)+1}$ computed in Corollary \ref{C:toric-family} 
(note that the arithmetic genus of an $n$-ribbon 
always satisfies $n-1\mid 2g$).
\end{remark}

\subsection{Computing the characters of $\delta_i$}  \label{subsection-delta}
In this section, we illustrate how the characters of line bundles $\delta_i$ can be computed.
 If $C$ is a curve with a $\GG_m$-action such that the discriminant locus inside 
 $\Def(C)$ is Cartier, then line bundles $\delta_i$ can be 
 defined in a neighborhood of $[C]$ in $\cU_g$.  The following proposition shows that the character of 
 $\delta_i$ is precisely minus the weighted degree of the discriminant.  

\begin{prop} \label{proposition-character-delta}
Let $C$ be a complete curve with miniversal deformation space $\Spf A$ and a $\GG_m$-action 
$\eta\co \GG_m \to \Aut(C)$.
Let $D$ be a Cartier divisor defined on a neighborhood $\cU_g$ in the stack of all complete 
genus $g$ curves.
Suppose that there is a cartesian diagram
$$\xymatrix{
V(f) \ar[r] \ar[d]		& \Spf A \ar[d] \\
D \ar[r]			& \cV
}$$
such that $f \mapsto \lambda^d f$ under the induced action of 
$\GG_m = \Spec k[\lambda, \lambda^{-1}]$ on $\Spf A$.  Then $$\chi_{\cL(D)}(C, \eta) = -d.$$
\end{prop}

\begin{proof}
Denote by $\sigma\co A \to k[\lambda, \lambda^{-1}] \hat{\tensor} A$ the dual action of
$\GG_m$ on $\Spf A$.  The exact sequence $0 \to \cL(-D) \to
\oh_{\cU_g} \to \oh_D \to 0$ restricted to $\Spf A$ corresponds to the
exact sequence
$$ 0 \to A_{\eta} \xrightarrow{f} A \to A/f \to 0$$
where $A_{\eta}$ is the $\GG_m$-$A$-module corresponding to the
character $\GG_m \xrightarrow{d} \GG_m$; that is, $A_{\eta}$ is $A$ as
an $A$-module with coaction $a \mapsto \lambda ^d \sigma(a)$.  Therefore
$\cL(-D)|_{B \GG_m}$ corresponds to the character $\GG_m
\xrightarrow{d} \GG_m$ and $\chi_{\cL(-D)}(C, \eta) = d$.
\end{proof}

We include the computation of the character of $\delta_i$ only for certain curves with $A_{2k}$ and $D_{2k+2}$
singularities but the same approach can be applied to a large class of curves; sampling of our computations 
is given in column four of Table \ref{T:table-characters}. 

The 
characters of $\delta_i$ will depend on the global geometry of the curve.  For instance, if an $A_{2k+1}$-singularity lies 
on a connected 
component intersecting the rest of the curve at two nodes, the character of $\delta_0$ will depend on whether the 
component is separating or not.  Furthermore, if an $A_{2k+1}$-singularity lies on a rational curve attached to the rest 
of the curve at one point, which we refer to as a ``dangling'' singularity (see Section \ref{section-dangling}), the value of 
$\delta_0$ will be different from the non-dangling case.  

\begin{example}[$A_{2k+1}$-singularity: non-separating case]
Let $C = C_0 \cup E$, where $C_0$ is a smooth curve of genus $g-k$, $E = E_1 \cup E_2$ is the union of two rational curves at the monomial\footnote{This simply means that $E_1\cup E_2$ is the projective closure of the affine curve
$y^2=x^{2k+2}$.} $A_{2k+1}$ singularity at $p$, and $E_i$ intersects $C_0$ at infinity in the node $q_i$.
The versal deformation space of $C$ can be written as
$$\Def(C) = \Def(C_0, q_1, q_2) \times \Cr(\hat{\oh}_{C,p}) \times \Def(\hat{\oh}_{C,p}) \times
\Def(\hat{\oh}_{C,q_1}) \times  \Def(\hat{\oh}_{C,q_2}),$$
where $\Cr(\hat{\oh}_{C,p})$ denotes the ``crimping'' deformations (see \cite{asw_crimping} for more details); the 
$\GG_m$-action on $\Cr(\hat{\oh}_{C,p})$ can be explicitly determined but doesn't affect this calculation.
We choose coordinates $a_{0},\dots, a_{2k}$ on $\Def(\hat{\oh}_{C,p})$ so that miniversal deformation of 
$\hat{\oh}_{C,p}$
is 
$$y^2 =x^{2k+2} + a_{2k}x^{2k} + \cdots + a_1 x +a_0,$$ 
and $n_i$ on $\Def(\hat{\oh}_{C,q_i})$ so that the miniversal deformation of $\hat{\oh}_{C,q_i}$
is $zw+ n_i=0$, where $z=1/x$ in the neighborhood of $\infty$ on $E_i$.
We have a one-parameter subgroup $\eta\co  \GG_m \to \Aut(C)$ such that $\GG_m$ acts by 
$\lambda\cdot (x,y)=(\lambda^{-1} x, \lambda ^{-k-1}y)$, 
$a_i \mapsto \lambda^{i-2k-2} a_i$, and $n_i \mapsto \lambda n_i$.  The
discriminant $\Delta\subset \Def(\hat{\oh}_{C,p})$ is given by the vanishing locus of the discriminant of the polynomial
 $x^{2k+2} + a_{2k}x^{2k}+\cdots+ a_0$. Thus it has weighted degree
$-(2k+1)(2k+2)$. The discriminant inside $\Def(\hat{\oh}_{C,q_i})$ is $\{n_i=0\}$  and has weighted degree 
$1$.  
Since $\delta_0 = V(\Delta) \cup V(n_1) \cup V(n_2)$, we conclude:
$$\chi_{\delta_0} = (2k+1)(2k+2)-2, \quad \text{and} \quad \chi_{\delta_i}
= 0 \text { for } i > 0.$$
\end{example}

\begin{example}[$A_{2k+1}$-singularity: separating case]
Let $C= C_0 \cup E$ be a curve as in the previous example with the exception that 
$C_0$ is now a {\em disconnected} curve with two connected components $C_1$ and $C_2$ of genera
$h_1$ and $h_2$, respectively. (Clearly, $g=h_1+h_2+k$.) 
Using the calculation of the previous example, we conclude from $\delta_0 = V(\Delta)$ 
and $\delta_{h_i} = V(n_i)$ that
$$\chi_{\delta_0} = (2k+1)(2k+2), \quad
\chi_{\delta_{h_1}} = \chi_{\delta_{h_2}} = -1, \quad
\chi_{\delta_{i}} = 0 \text { for } i \neq 0, h_1, h_2.$$
\end{example}

\begin{example}[$A_{2k+1}^{\{1\}}$-singularity: dangling case]
Let $C= C_0 \cup E$, where $C_0$ is a smooth curve of genus $g-k$ , $E = E_1 \cup E_2$ is the union of two rational curves at the monomial $A_{2k+1}$ singularity, and $E_1$ intersects $C_0$ at a node.  Then
$$\chi_{\delta_0} = (2k+1)(2k+2), \quad
\chi_{\delta_{k}}  = -1, \quad
\chi_{\delta_{i}} = 0 \text { for } i \neq 0, k.
$$
\end{example}

\begin{example}[$A_{2k+1}^{\{\}}$-singularity: isolated case]
Let $C=  E_1 \cup E_2$ be the union of two rational curves at the monomial $A_{2k+1}$ singularity.  Then
$$\chi_{\delta_0} = (2k+1)(2k+2), \quad
\chi_{\delta_{i}} = 0 \text { for } i > 0.
$$
\end{example}

\begin{example} [$D_{2k+2}$-singularity: non-separating case] 
Let $C = C_0 \cup E$, where $C_0$ is a genus $g-k$ curve, $E = E_1 \cup E_2 \cup E_3$ is the union of three rational curves at the monomial $D_{2k+2}$ singularity at $p \in E$, with $E_2$ and $E_3$ tangent, and each $E_i$ intersects $C_0$ at a node $q_i$.  We write
$$ \Def(C) = \Def(C_0, q_1, q_2, q_3) \times \Cr(\hat{\oh}_{C,p}) \times \Def(\hat{\oh}_{C,p}) \times \prod_{i=1}^3\Def(\hat{\oh}_{C,q_i}).$$
We can choose coordinates so that
$$\begin{aligned}
\Def(\hat{\oh}_{C,p}) &= \{xy^2=x^{2k+1} + a_{2k-1}x^{2k-1} + \cdots + a_1 x + a_0 + b y \}, \\
\Def(\hat{\oh}_{C,q_i}) &= \{z_iw_i+ n_i=0\}, \\
\end{aligned}$$
where $z_1=1/y$ and $z_2=z_3=1/x$ near $\infty$ on $E_1$, and $E_2, E_3$, respectively.
We have a one-parameter subgroup $\eta\co \GG_m \to \Aut(C)$ such that $\GG_m$ acts via 
$\lambda\cdot(x,y)=(\lambda ^{-1}x, \lambda ^{-k}y)$, and 
$$a_i \mapsto \lambda ^{i-2k-1}a_i \qquad b \mapsto \lambda ^{-k-1} b \qquad n_1 \mapsto \lambda ^{k} n_1 \qquad n_2 \mapsto \lambda n_2 \qquad n_3 \mapsto \lambda  n_3$$
The discriminant $\Delta\subset \Def(\hat{\oh}_{C,p})$ has weight $(2k+1)(2k+2)$, so we conclude that
$$\chi_{\delta_0} = (2k+1)(2k+2)-(k+2), \quad  \chi_{\delta_i}= 0 \text { for } i > 0.$$
\end{example}

\begin{example} [$D_{2k+2}^{\{1,2\}}$-singularity]
Let $C = C_0 \cup E$, where $C_0$ is a genus $g-k$ curve, $E = E_1 \cup E_2 \cup E_3$ is the union of three rational curves at the monomial $D_{2k+2}$ singularity, with $E_2$ and $E_3$ tangent, and $E_1$ and $E_2$ meet 
$C_0$ in nodes.  Using the calculation above, we conclude that
$\chi_{\delta_0} = (2k+1)(2k+2)-(k+1)$ and $\chi_{\delta_i}
= 0$ for $i \neq 0$.
\end{example}

\begin{example} [$D_{2k+2}^{\{1,2\}}$-singularity]
Let $C$ be a curve as in the previous example except that only the branch $E_1$ intersects $C_0$.  Then 
$\chi_{\delta_0} = (2k+1)(2k+2)-k$ and $\chi_{\delta_i}
= 0$ for $i \neq 0$.
\end{example}

\subsection{Computing the characters of $K$}

\begin{lem}  Let $C$ be a curve with a smooth deformation space and $\Aut(C)^\circ$ abelian.  Let $\eta \co \GG_m \to \Aut(C)$ be a one-parameter subgroup.  The character
$\chi_{K}(C, \eta)$ is the character of the determinant of $T^1(C)$ given by the natural $\GG_m$-action.
\end{lem} 

\begin{proof} With the notation from Section \ref{section-discussion-K},  
choose a presentation $\cM = [\Hilb / \PGL_{N+1}]$.  Fix a closed immersion $C
\hookarr \PP^N$; this determines an element $x=[C \hookarr \PP^N]$ of $\Hilb$.  By considering the dual of the 
cotangent complex $L_{\cM}$, we arrive at an exact sequence
$$0 \to H^0(L_{\cM}^{\dual}) \to \mathfrak{g} \tensor \oh_{\Hilb} \to
T_{\Hilb} \to H^1(L_{\cM}^{\dual}) \to 0$$
of sheaves on $\Hilb$.  By restricting this sequence to $x$, we obtain an exact sequence
$$ 0 \to \mathfrak{g}_x \to \mathfrak{g} \to H^0(C, N_{C/\PP^N}) \to
T^1(C) \to 0$$
of $G_x$-representations.  The morphism $\mathfrak{g} \to H^0(C, N_{C/\PP^N})$ is obtained by
differentiating the map $G \to \Hilb, g \mapsto g \cdot x$.  Since the adjoint action on $\mathfrak{g}_x$ is trivial, we obtain the result.
\end{proof}

We will use the above lemma to compute the character of $K$ in one particular example.  Other examples can be computed similarly.  

\begin{example}[$A_{2k}$-singularity]
Let $C=C_{0} \cup E$, where $E$ is a smooth rational curve with a higher cusp $y^2=x^{2k+1}$ at $p=0$, and a nodal attachment to $C_{0}$ at $q=\infty$.   The first order deformation space can be written as
$$T^1(C) = T^1(C_0, q) \times \Cr(\hat{\oh}_{C,p}) \times T^1(\hat{\oh}_{C,p}) \times
T^1(\hat{\oh}_{C,q}),$$
where $\Cr(\hat{\oh}_{C,p})$ denotes the ``crimping'' deformations (see \cite{asw_crimping} for more details).  We can choose coordinates
$$\begin{aligned}
T^1(\hat{\oh}_{C,p}) &= \{ y^2 - x^{2k+1} + a_{2k-1}x^{2k-1} + \cdots + a_1 x +
a_0 = 0 \}, \\
T^1(\hat{\oh}_{C,q}) &= \{zw+ n=0\},
\end{aligned}$$
and a one-parameter subgroup $\eta\co  \GG_m \to \Aut(C)$ acting
via $\lambda\cdot (x,y)= (\lambda ^{-2}x, \lambda ^{-(2k+1)}y)$. 
Then $a_i \mapsto \lambda ^{2i-4k-2} a_i$ and  $n \mapsto \lambda n$. 
Therefore, the character of $T^1(\hat{\oh}_{C,p})$ is 
$$-(4+6+\cdots + (4k+2)) = -(4k^2+6k).$$  The character of $T^1(\hat{\oh}_{C,q})$ is $1$.  The character of $ T^1(C_0, q)$ is trivial.  For $k \ge 2$, by \cite[Proposition 3.4]{asw_crimping}, the weights of the action on $\Cr(\hat{\oh}_{C,p})$ are $1,3, \ldots, 2k-3$.  Therefore, the character of $\Cr(\hat{\oh}_{C,p})$ is $(k-1)^2$.  It follows that $\chi_K = -3k^2-8k+2$.  As a reality check, one sees that indeed $\chi_K = 13 \chi_{\lambda} - 2 \chi_{\delta}$ (consult Table 
\ref{T:table-characters} for values of $\chi_{\lambda}$ and $\chi_{\delta}$).
\end{example}

\section{Character theory vs intersection theory}

Let $C$ be a complete curve of arithmetic genus $g$
with a $\GG_{m}$-action $\eta\co \GG_{m}\ra \Aut(C)$. 
A \emph{versal deformation space with $\GG_m$-action} is a pointed affine scheme $0\in \Def(C)$ with a 
$\GG_m$-action together with a smooth morphism $[\Def(C) / \GG_m] \to \cV_g$ to the 
{\em stack of all complete curves of arithmetic genus $g$} sending $0$ to $[C]$.  
If $\cL$ is a line bundle defined on $\cV_g$ in a neighborhood of $C$ and $\Def(C)$ is a versal deformation space with $\GG_m$-action, then after shrinking $\Def(C)$, we may assume that $\cL$ is defined on $[\Def(C)/\GG_m]$.

\begin{remark} If $\omega_C$ is ample, the versal deformation space of $C$ is normal, and $\Aut(C)$ is linearly 
reductive, then it follows from \cite[Theorem 3]{alper_quotient} that there exists a versal deformation space with 
$\GG_m$-action (see also \cite[Proposition 2.3]{pinkham} for the formal case).
\end{remark}

We specialize to the case when $C$ 
has an isolated singularity $p\in C$ 
and $\hat{\cO}_{C,p}$ is positively 
graded by the $\GG_m$-action.  
Then by Pinkham's theory of deformations of varieties with $\GG_m$-action \cite[Proposition 2.2]{pinkham},
the space of infinitesimal deformations of $C$ has a decomposition
into the weight spaces: 
$$
T^1_C=\bigoplus_{\nu=-\infty}^{\infty} T^1_C(\nu).
$$
Following Pinkham \cite[Section (3.1)]{pinkham}, we define $\Def^{\, -}(C)$ to be the closed subscheme of
$\Def(C)$ corresponding to {\em negative deformations}: The tangent space to $\Def^{-1}(C)$ is 
$\bigoplus_{\nu<0} T^1_C(\nu)$ and the coordinate ring of $\Def^{-}(C)$ is positively graded.
The relationship between intersection theory on $[\Def^{\, -}(C) /\GG_m]$ 
and characters is given by the following observation.
\begin{thm}   \label{theorem-character-intersection}
Let $C$ be a complete curve of arithmetic genus $g$ 
and $0 \in \Def(C)$ be its versal deformation space with $\GG_m$-action.  Let $B$ be any complete curve with
a non-constant map $B \to [\Def^{-} (C)/\GG_m]$ and let $\cL$ be a line bundle on $\cV_g$.   Then
$$\chi_{\cL}(C, \eta) = - \frac{ \cL\cdot B } { \deg(B)},$$
where $\deg(B)$ is the degree of $B$ with respect to the natural $\oh(1)$ on $[\Def^{\, -}(C) / \GG_m]$.  In particular, if $C$ is Gorenstein (resp. the discriminant locus in $\Def(C)$ is Cartier), then 
 $$\chi_{\lambda_i}(C, \eta) = - \frac{ \lambda_i \cdot B} { \deg(B)} \qquad 
 \left(\text{resp. } \chi_{\delta_i}(C, \eta) =  - \frac{\delta_i \cdot B} { \deg(B)}  \right).$$
\end{thm}

\begin{proof}
We can write $\Def^{\, -}(C) = \Spec A$ with $A$ a positively graded $k$-algebra. The line bundle $\cL$ 
corresponds to a graded projective
 $A$-module which is free of rank $1$ by \cite[Theorem 19.2]{eisenbud}. It follows 
 that $\cL = \tilde{A(d)} = \oh(d)$ for some $d$.
Therefore $\chi_{\cL}(C, \eta) = -d$ 
and $\cL \cdot B = \deg_B(\cL) =d\deg(B)$.
\end{proof}
Theorem \ref{theorem-character-intersection} allows us to compute characters via intersection theory on
one-parameter families of {\em stable} curves so long as the locus of stable curves 
inside $[\Def^{-}(C)  / \GG_m]$ contains complete one-parameter families. This is not an uncommon 
occurrence since, as Pinkham shows (in the case of unibranch singularities), 
$[\Def^{\, -}(C) / \GG_m]$ contains an open set parameterizing {\em smooth} curves of genus $g$
\cite[Theorem 1.15]{pinkham}.

For special classes of planar singularities even stronger statement holds. For example, if $\hat{\cO}_{C,p}$
is an ADE singularity, then $\Def^{\, -}(C)=\Def(C)$ and the locus of worse-than-nodal curves in 
$[\Def^{\, -}(C) /\GG_m]$ is of codimension two. It follows that the 
characters of ADE singularities can all be computed by writing down a complete one-parameter family
of stable curves
in $[\Def^{\, -}(C)/\GG_m]$ and computing degrees of line bundles $\lambda$ and $\delta$
using standard intersection theory. We do so in a number of cases in Section \ref{S:intersection-theory}.

In the other direction, Theorem \ref{theorem-character-intersection} suggests a possibility of computing 
slopes of special loci inside $\Mg{g}$: 
\begin{example}[Toric singularities] Consider the planar toric singularity $C: x^p-y^q=0$. Its miniversal
deformation is 
\begin{equation*}
x^p=y^q+\sum a_{ij} x^iy^j, \quad 0\leq i\leq p-2, \ 0\leq  j\leq q-2.
\end{equation*}
We have that $\Def^{-}(C)=\spec k[a_{ij}: qi+pj<pq]$. The resulting weighted projective stack
$[(\Def^{-}(C)\setminus 0)/ \GG_m]$
is a moduli space of curves on $\PP(q,p,1)$ defined by the weighted homogeneous equation
\begin{equation}\label{E:toric-homogeneous}
x^p=y^q+\sum a_{ij} x^iy^jz^{pq-qi-pj}, \quad 0\leq i\leq p-2, \ 0 \leq j\leq q-2, \ qi+pj<pq.
\end{equation}
Theorem  \ref{theorem-character-intersection} implies that for any complete family of stable curves
$B\ra [(\Def^{-}(C)\setminus 0)/ \GG_m]$, the slope of $B$ is 
$(\delta\cdot B)/(\lambda\cdot B)=\chi_{\delta}(C)/\chi_{\lambda}(C)$.
By considering the monomial unibranch singularities $y^3=x^{3k+1}$ and $y^3=x^{3k+2}$, we recover
the following result of Stankova \cite{stankova} (see also Remark \ref{R:anand}).
\begin{cor}\label{C:stankova} For $g\equiv 0,1 \mod 3$, there is a complete family $B$ of generically smooth 
genus $g$ stable trigonal curves such that $(\delta\cdot B)/(\lambda\cdot B)=36(g+1)/(5g+1)$.
\end{cor}
\begin{proof}
Let $C$ be the monomial unibranch singularity $y^3=x^{3k+1}$. 
From Equation \eqref{E:toric-homogeneous} 
the restriction of its miniversal deformation to $\Def^{\, -}(C)$ is 
\begin{align}\label{E:trigonal}
y^3=x^{3k+1}+y(a_{2k}x^{2k}+\cdots+a_0)+(b_{3k-1}x^{3k-1}+\cdots+b_0).
\end{align}
It follows that
$[(\Def^{\, -}(C) \setminus 0  ) / \GG_m]\simeq \mathcal{P}(2,5,\dots,6k+2,6,9,\dots,9k+3)$ is a
moduli space of trigonal curves of genus $g=3k$ defined by Equation \eqref{E:trigonal} on
$\PP(3,3k+1,1)$. 

Evidently, there is a complete family 
$B\ra [(\Def^{\, -}(C) \setminus 0  ) / \GG_m]$ of at-worst nodal irreducible curves. 
Applying Theorem \ref{theorem-character-intersection} to this family and using 
computations of Example \ref{E:monomial-unibranch}, 
we find that 
$\lambda\cdot B=\chi_{1}(C)=2g(5g+1)/12$ and $\delta\cdot B=13\chi_{1}(C)-\chi_2(C)=6g(g+1)$. 
This gives slope $36(g+1)/(5g+1)$.

Considering $y^3=x^{3k+2}$, we obtain in an analogous fashion a complete family of trigonal
curves of genus $g=3k+1$ with slope $36(g+1)/(5g+1)$.
\end{proof}
We note that in contrast to a simple construction above, an extremal family achieving 
slope $36(g+1)/(5g+1)$ is obtained by a laborious 
construction in \cite{stankova}. However, our methods do not establish a stronger result, 
also due to Stankova, 
that $36(g+1)/(5g+1)$ is {\em maximal possible} among slopes of trigonal families 
of genus $g$. 
\end{example}

\section{Intersection theory computations}
\label{S:intersection-theory}

In this section, we use intersection theory to find slopes of one-parameter families of curves arising from 
stable reduction of certain planar singularities. Namely, we treat 
the cases of $A_{2k+1}$, $A_{2k+1}^{\{1\}}$, $D_{2k+2}$, $D_{2k+2}^{\{1,2\}}$, 
and toric singularities $x^p=y^q$.
Our computations agree with the results obtained in Section \ref{S:character-computations}.
That this should be the case follows from Theorem \ref{theorem-character-intersection} after 
verification that every family we write down comes from $[\Def^{-}(C)/\GG_m]$ of an appropriate
 singular curve $C$.
The same techniques can be applied to other singularities. However, as singularities become more complicated,
the task of writing down a complete family of stable limits becomes substantially more challenging. 

\subsection{Hyperelliptic tails, bridges and triboroughs}
\label{hyperelliptic}
\begin{example}[Hyperelliptic bridges]
\label{E:hyperelliptic-bridges}
We construct a complete one-parameter family $B_k$ of $2$\nb-pointed stable hyperelliptic curves of genus $k$, 
with marked points conjugate under the hyperelliptic involution, that arises from 
stable reduction of $A_{2k+1}$ singularity. It follows from our construction and \cite[Theorem 6.5]{hassett-stable}
that $B_k$ comes from $[\Def^{-1}(C)/\GG_m]$ where $C$ is the projective closure of $y^2=x^{2k+2}$.
We show that $B_k$ intersects divisors on $\Mg{k,2}$ as follows:
\begin{equation*}
\begin{aligned}
\lambda\cdot B_k &=(k^2+k)/2, \\
\delta_{0}\cdot B_k &=(2k+1)(2k+2), \\
\psi_1\cdot B_k &=\psi_2\cdot B_k=1, \\
\delta_1\cdot B_k &=\cdots =\delta_{\lfloor k/2\rfloor} \cdot B_k=0.
\end{aligned} 
\end{equation*}
If $B_k \ra \Mg{g}$ is the family of hyperelliptic genus $k$ bridges 
obtained by attaching a constant genus $(g-k-1)$ curve to the marked sections, then
$(K_{\Mg{g}}+\alpha\delta)\cdot B_k\leq 0$ exactly for $\alpha\leq (3k+11)/(8k+12)$. This of course 
agrees with the character theory 
computation of the $\alpha$-value of $A_{2k+1}$-singularity (see Table \ref{T:table-characters}) due to Theorem 
\ref{theorem-character-intersection}.

To construct $B_k$, 
take a Hirzebruch surface 
 $\mathbb{F}_1 \ra B$ over $B\simeq \PP^1$. Denote by 
$E$ the unique $(-1)$\nb-section and by $F$ the fiber. Next, 
choose $2k+2$ general divisors $S_1,\dots, S_{2k+2}$ in the linear system 
$\vert E+F\vert$ (these are sections of $\mathbb{F}_1\ra B$ of self-intersection $1$).
The divisor $\sum_{i=1}^{2k+2} S_i$ is divisible by $2$ in $\Pic(\mathbb{F}_1)$ 
and so there is a cyclic degree $2$
cover $\pi\co X\ra \mathbb{F}_1$ branched over $\sum_{i=1}^{2k+2} S_i$.
We have that $\pi^{-1}(E)=\Sigma_1+\Sigma_2$ is a disjoint union of two sections. 
Thus, $\pi\co X \ra B$ is a family of at-worst nodal hyperelliptic curves of genus $k$ with two conjugate sections
$\Sigma_1$ and $\Sigma_2$.

From the construction, there are $\binom{2k+2}{2}$ nodes in the fibers of $\pi$. Because $X$ has 
$A_1$ singularity at each of these nodes, we have
$$\delta_{X/B}=2 \binom{2k+2}{2}=(2k+1)(2k+2).$$ 
From $K_{X/B}=\pi^*(K_{\mathbb{F}_1/B}+\frac{1}{2}\sum_{i=1}^{2k+2} S_i)$ we deduce that 
$$12\lambda_{X/B}-\delta_{X/B}=K_{X/B}^2=2(K_{\mathbb{F}_1}+2F+(2k+2)(E+F))^2=2(k^2-1).$$
Finally, the self-intersection of each $\Sigma_i$
is $(-1)$. It follows that $\pi\co X\ra B$ is the requisite family.
\end{example}

\begin{example}[Hyperelliptic tails attached at arbitrary points] 
We now consider a family of tails appearing from stable reduction of a dangling 
$A_{2k+1}^{\{1\}}$ singularity (see Section \ref{S:dangling}).
Taking family $B_k$ constructed in Example \ref{E:hyperelliptic-bridges} above and forgetting one 
marked section, we arrive at a family $H_k\subset \Mg{k,1}$ of hyperelliptic curves with a single marked section. Furthermore, 
\begin{equation*}
\begin{aligned}
\lambda\cdot H_k &=(k^2+k)/2, \\
\delta_{0}\cdot H_k &=(2k+1)(2k+2), \\
\psi\cdot H_k &=1, \\
\delta_1\cdot H_k &=\cdots =\delta_{\lfloor g/2\rfloor} \cdot H_k=0.
\end{aligned} 
\end{equation*}
In particular, the locus of curves with a hyperelliptic genus $k$ tail
falls in the 
base locus of $K_{\Mg{g}}+\alpha\delta$ for $\alpha<(3k^2+11k+4)/(8k^2+12k+2)$. For example,
when $k=2$, this shows that $\Delta_2\subset \Mg{g}$ is covered by curves on which 
$K_{\Mg{g}}+(19/29)\delta$ has degree $0$.
\end{example}

\begin{example}[Hyperelliptic triboroughs]
\label{E:hyperelliptic-triboroughs}
Next, we construct a complete one-parameter family $Tri_k$ of $3$\nb-pointed stable 
hyperelliptic curves of genus $k$, with two marked points conjugate, that arises
from stable reduction of $D_{2k+2}$ singularity. It is easy to verify that this family comes
from $[\Def^{-1}(C)/\GG_m]$ where $C$ is the projective closure of $x(y^2-x^{2k})=0$. 
We show that $Tri_k$ intersects divisor classes on $\Mg{k,3}$ as follows:
\begin{equation*}
\begin{aligned}
\lambda\cdot Tri_k &=k^2+k, \\
\delta_{0}\cdot Tri_k &=2(2k+1)(2k+2), \\
\psi_1\cdot Tri_k &=\psi_2\cdot Tri_k=2, \\
\psi_3\cdot Tri_k &=2k, \\
\delta_1\cdot Tri_k &=\cdots =\delta_{\lfloor k/2\rfloor} \cdot Tri_k=0.
\end{aligned} 
\end{equation*}

\noindent
The construction of $Tri_k$ parallels that of the 
family $B_k$ above. 
Namely, keeping the notation of Example \ref{E:hyperelliptic-bridges}, consider an additional section $S_0$ of $\mathbb{F}_1$ 
of self-intersection $1$ such that $S_0$ is transverse to $\sum_{i=1}^{2k+2} S_i$.
Set $C:=\pi^{-1}(S_0)$. Then $C$ is a 
degree $2$ cover of $B$ (of genus $k$). Note that $C^2=2$ on $X$.
Consider the base extension $\pi'\co Y:= X\times_{B} C \ra C$. The preimage of $C$ on $Y$ is the union of two sections $C_1$ and $C_2$, intersecting transversally in 
$2k+2$ points. By construction,
$(C_1+C_2)^2=2C^2=4$, and so $C_1^2=C_2^2=-2k$. Setting $\Sigma_3:=C_1$, we obtain
the requisite family $\pi'\co Y\ra C$ of hyperelliptic genus $k$ curves with two conjugate sections 
(the preimages of $\Sigma_1$ and $\Sigma_2$)
of self-intersection $(-2)$ and the third section $\Sigma_3$ 
of self-intersection $(-2k)$. 
\end{example}

\begin{example}[Hyperelliptic bridges attached at arbitrary points]
We now consider a family of tails appearing in stable reduction of a dangling 
$D_{2k+2}^{\{1,2\}}$ singularity (see Section \ref{S:dangling}).
Taking family $Tri_k$ constructed in Example \ref{E:hyperelliptic-triboroughs} above and forgetting one 
conjugate section, we arrive at a family $BH_k$ of hyperelliptic curves with two marked points. 
The intersection numbers of $BH_k$ are
\begin{equation*}
\begin{aligned}
\lambda\cdot BH_k &=k^2+k, \\
\delta_{0}\cdot BH_k &=2(2k+1)(2k+2), \\
\psi_1\cdot BH_k &=2, \quad \psi_2\cdot BH_k =2k, \\
\delta_1\cdot BH_k &=\cdots =\delta_{\lfloor k/2\rfloor} \cdot BH_k=0.
\end{aligned} 
\end{equation*}

In particular, the locus of curves with a hyperelliptic bridge of genus $k$ 
attached at arbitrary points falls in the 
base locus of $K_{\Mg{g}}+\alpha\delta$ for $\alpha<(3k^2+7k+4)/(8k^2+10k+2)$. For example,
when $k=2$, this shows that the locus of curves with genus $2$ bridges in $\Mg{g}$ 
is covered by curves on which $K_{\Mg{g}}+(5/9)\delta$ has degree $0$.
\end{example}

\subsection{Toric singularities}
\label{S:toric}

How can we write down a complete one-parameter family of stable
limits of a singularity in such a way that its
intersection numbers
with divisors on $\Mg{g}$ can be computed? We give a complete
answer only in the case of a planar toric singularity $x^{p}=y^{q}$,
even though our method applies more generally to any planar singularity.

Our approach is via degenerations: Begin with a complete family $F_1$
of at-worst nodal curves -- a general pencil of plane curves of degree
$d\gg 0$ will do.
Now vary $F_1$ in a one-parameter family $F_s$ in such a way that
among curves in $F_0$
exactly one curve $C$ has singularity $f(x,y)=0$ while the rest
are at-worst nodal.
Since the generic points of $F_0$ and $F_1$ are smooth curves of genus
$g=\binom{d-1}{2}$, we obtain two $1$-cycles $F_0, F_1\in \N_1(\Mg{g})$.
For a line bundle
$\L\in \Pic(\Mg{g})$ the numbers  $\L\cdot F_0$ and $\L \cdot F_1$
will differ. If we denote
by $\F$ the total space of $\{F_s\}$, then
the discrepancy between $\L\cdot F_0$ and $\L\cdot F_1$ is accounted
for by indeterminacy
of the rational map $\F \dashrightarrow \Mg{g}$ at the point $[C]$. In fact, if
\[
\xymatrix{
&W \ar[rd]^{h} \ar[ld]_{f}&\\
\F \ar@{-->}[rr]&&\Mg{g}
}
\]
is the graph of this rational map, then in $\N_1(\Mg{g})$ we have
$F_1=F_0+h(f^{-1}([C]))$.
By construction, $Z:=h(f^{-1}([C]))$ is a $1$\nb-cycle inside the variety of stable limits of $f(x,y)=0$. 
The slope of $Z$ is then given by
$
(\delta\cdot F_1 -\delta\cdot F_0)/(\lambda\cdot F_1 -\lambda\cdot F_0).
$

We now perform the necessary computations for toric planar singularities. 
To begin, let $C$ be a plane curve of degree
$d\gg 0$, with an isolated singularity $x^{pb}=y^{qb}$, where $p$ and $q$ are coprime.
The possible stable limits of $C$ have the following description due to Hassett:
\begin{prop}[{\cite[Theorem 6.5]{hassett-stable}}]
\label{P:tails-description}
The stable limits of $C$ are
of the form $\tilde{C} \cup T$,
where the tail $(T,p_1,\dots, p_b)$ is a $b$-pointed curve of arithmetic genus
$g=(pqb^2-pb-qb-b+2)/2$. Moreover,
\begin{enumerate}
\item $K_T=(pqb-p-q-1)(p_1+\dots+p_b)$.
\item $T$ is $qb$-gonal with $g^1_{qb}$ given by $|q(p_1+\dots+p_b)|$.
\end{enumerate}
\end{prop}
Given a two-parameter deformation 
of the curve $C$, one obtains a one-dimensional family of stable limits. The next proposition constructs
one such family and computes its intersections with the divisor classes $\lambda, \delta$, and $\psi$ on
$\Mg{g,b}$.
\begin{prop}\label{P:toric-family} Let $(p,q)=1$.
Suppose $F_0$ is a pencil of plane curves of degree $d\gg 0$ containing 
a curve $C$ with a unique singularity $x^{pb}=y^{qb}$ and such that the total family over
$F_0$ is smooth. Consider 
a deformation $\F=\{F_s\}$ of $F_0$ such that $F_1$ is a general pencil. 
If $\F \stackrel{f}{\longleftarrow} W \stackrel{h}{\longrightarrow} \Mg{\binom{d-1}{2}}$ is the graph of the rational 
map $\F\dashrightarrow \Mg{\binom{d-1}{2}}$,
 then
the $1$-cycle
$Z:=h(f^{-1}([C]))$ inside $\Mg{g, b}$ (here, $g=(pqb^2-pb-qb-b+2)/2$) is irreducible and 
satisfies:
\begin{multline*}
\begin{aligned}
\lambda\cdot Z &=\frac{b}{12}\left( (pqb-p-q)^2+pq(pqb^2-pb-qb+1)-1\right),\\
\delta_{0}\cdot Z &=pqb(pqb^2-pb-qb+1), \\
\psi\cdot Z &=b.
\end{aligned}
\end{multline*}
\end{prop}

\begin{proof} 
Without loss of generality, we can assume that the total space of the family of plane curves over $\F=\{F_s\}$ 
has local equation
$x^p=y^q+sxy+t$. Then the simultaneous stable reduction of this family is obtained by
the weighted blow-up of $\AA^2_{s,t}$ with weights $w(s,t)=(pq-p-q,pq)$. It follows that 
$Z\simeq \PP(pq-p-q,pq)$ is irreducible. For $[s:t]\in Z$, the stable
curve over $[s:t]$ is a curve on $\PP(q,p,1)$ defined by the weighted homogeneous equation 
$$x^p=y^q+sxyz^{pq-p-q}+tz^{pq}=0.$$ It is easy to verify that every member of this
family is in fact an irreducible stable curve. 

We proceed to compute the intersection numbers of $Z$ with divisor classes $\lambda, \delta_0$, 
and $\psi$. 
Let $\X_i$ be the total families of pencils $F_i$ ($i=0,1$).
Our first goal is to compare the degrees of $\delta$ and $\kappa$
on $\X_0$ and $\X_1$: 

Since $F_1$ is a general pencil, we have
$\delta(\X_1)=\delta_0(\X_1)=3(d-1)^2$. 
To find the number of singular fibers in $\X_0\smallsetminus C$, we observe
that the topological Euler characteristic of $C$ is
$
2-2g(C)-(b-1),
$
where $g(C)=\binom{d-1}{2}-g-b+1$ is the geometric genus of $C$. Since topological Euler characteristics
of $\X_0$ and $\X_1$ are the same, we have that
\[
\delta_{0}(\X_0 \smallsetminus C)=\delta_0(\X_1)-(2g+b-1)=\delta_0(\X_1)-(pqb^2-pb-qb+1).
\]
Since $\X_0$ and $\X_1$ are two families of plane curves of degree $d$, 
we have
\[
\kappa(\X_0)=\kappa(\X_1).
\]

Next, to compare intersection numbers of $F_0$ and $F_1$ with $\lambda$ and
$\delta_0$, we
need to write down a family of stable curves over each $F_i$. There is
nothing to do in the
case of $F_1$, since it is already a general pencil of plane curves of degree $d$. In particular, we have
$\lambda\cdot F_1=(\kappa(\X_1)+\delta_0(\X_1))/12$ by Mumford's formula.
To write down a stable family over $F_0$, we perform stable reduction of $\X_0\ra F_0$ in two steps:

\noindent
{\em Base change:} To begin, make a base change of order $bpq$
to obtain the family $\cY$ with local equation $x^{pb}-y^{qb}=t^{pqb}$.

The numerical invariants of $\cY$ are
\begin{align*}
\kappa(\cY) &=pqb\kappa(\X_0)=pqb\kappa(\X_1), \quad \text{and}\\
\delta_0(\cY\smallsetminus C)
&=pqb\delta_0(\X_0\smallsetminus C)=pqb(\delta_{0}(\X_1)-(pqb^2-pb-qb+1)).
\end{align*}

\noindent
{\em Weighted blow-up:} Let $\cZ$ be the
weighted blow-up of $\cY$, centered at $x=y=t=0$, with weights $w(x,y,t)=(q,p,1)$. The central fiber of $\cZ$ 
becomes 
$\tilde C \cup T$ of the form described in Proposition \ref{P:tails-description}, with smooth $T$.
The self-intersection of the tail $T$ inside $\cZ$ is $(-b)$. By intersecting both sides of 
$K_{\cZ}=\pi^*K_{\cY}+aT$ with $T$, we find that $a=p+q-pqb$. It follows that
\[
\kappa(\cZ)=\kappa(\cY)-b(pqb-p-q)^2=pqb\kappa(\X_1)-
b(pqb-p-q)^2.
\]

The number $\delta_{0}(\cZ\setminus (\tilde C \cup T))$ of singular fibers in $\cZ\setminus (\tilde C \cup T)$ is the same as in $\cY \smallsetminus C$ 
and
equals to
\[
pqb\delta_{0}(\X_1)-pqb(pqb^2-pb-qb+1).
\]
Remembering that the central fiber of $\cZ$ has exactly $b$ nodes, we
compute
\begin{align*}
\delta_0\cdot Z=pqb\delta_0(\X_1)-\delta_0(\cZ) &=pqb(pqb^2-pb-qb+1)-b, \\
\kappa \cdot Z=pqb\kappa(\X_1)-\kappa(\cZ) &=b(pqb-p-q)^2.
\end{align*}
Using Mumford's formula $\lambda=(\kappa+\delta)/12$, we obtain
\begin{align*}
\lambda\cdot Z=\frac{b}{12}\left(
(pqb-p-q)^2+pq(pqb^2-pb-qb+1)-1\right).
\end{align*}
We leave it as an exercise for the reader to verify that $\psi\cdot Z=b$.

\end{proof}

\begin{cor}\label{C:toric-family} Suppose $p$ and $q$ are coprime.
Then for the one-parameter family $Z$ of irreducible
one-pointed tails of stable limits of $x^p=y^q$, constructed in Proposition \ref{P:toric-family}, we have
\[
\frac{(\delta-\psi)\cdot Z}{\lambda\cdot
Z}=12\frac{pq(p-1)(q-1)-1}{(p-1)(q-1)(2pq-p-q-1)}.
\]
Considered as a family of unpointed curves of genus $g=(p-1)(q-1)/2$, $Z$ has slope
\[
\frac{\delta \cdot Z}{\lambda\cdot Z}=\frac{12pq}{2pq-p-q-1}.
\]
\end{cor}

\section{Connections to GIT} 
\label{section-git}

The $\alpha$-invariant can be reinterpreted in terms of the Hilbert-Mumford index in geometric invariant theory. 
Recall that for every $g$, $n$, and $m$, we have the Hilbert and Chow GIT quotients
$$\overline{\Hilb}_{g,n}^{\, m, \ss} \gitq \SL_{N}  \quad \text{and} \quad \overline{\textrm{Chow}}_{g,n}^{\, \ss} \gitq \SL_{N}$$
parameterizing, respectively, semistable $m^{th}$ Hilbert points of $n$-canonically
embedded curves of genus $g$, and semistable Chow points of $n$-canonically 
embedded curves of genus $g$ curves, up to projectivities. 
Here, $N = g$ if $n=1$, and $N = (2n-1)(g-1)$ if $n > 1$.

\begin{prop}  \label{proposition-git}
Let $C$ be a Gorenstein $n$-canonically embedded genus $g$ curve which admits a $\GG_m$-action 
$\eta\co \GG_m \to \Aut(C)$. Consider the induced one-parameter subgroup 
$\tilde{\eta}\co \GG_m \to \SL_{N}$.  Then the Hilbert-Mumford indices of the $m^{th}$ Hilbert 
point of $C$, respectively the Chow point of $C$, with respect to $\tilde{\eta}$ are
$$
\mu^{\overline{\Hilb}_{g,n}^{\, m}} ([C], \tilde{\eta})= 
 \begin{cases}
 \chi_{\lambda}+(m-1)\left[ ((4g+2)m-g+1) \chi_{\lambda}-\frac{gm}{2} \chi_{\delta} \right], & \ \text{if $n=1$}, \\
 (m-1)(g-1)\left[ (6mn^2-2mn-2n+1) \chi_{\lambda}-\frac{mn^2}{2} \chi_{\delta}  \right], & \ \text{if $n>1$,}\\
 \end{cases}
$$
and
$$\mu^{\overline{\text{{\em Chow}}}_{g,n}} ([C], \tilde{\eta})= 
\begin{cases}
	(4g+2) \chi_{\lambda} - \frac{g}{2} \chi_{\delta}\ , & \text{ if } n = 1,\\
	(g-1)n[(6n-2) \chi_{\lambda} - \frac{n}{2} \chi_{\delta}], & \text{ if } n > 1.
\end{cases}
$$
\end{prop}

\begin{proof} This result follows directly by computing the divisor classes of the GIT linearizations as in \cite[Theorem 5.15]{mumford-stability} or \cite[Section 5]{hassett-hyeon_flip}.
\end{proof}

This proposition implies that if one can compute the characters of $\lambda$ and $\delta$ (or equivalently $\lambda$ 
and $\lambda_2$) with respect to one-parameter subgroups of the automorphism group, then one immediately knows 
the Hilbert-Mumford indices for all Hilbert and Chow quotients for such one-parameter subgroups. Moreover, 
if the Hilbert-Mumford index with respect to a one-parameter subgroup of the automorphism group is non-zero, then 
the curve is unstable. In particular, we recover results of \cite[Propositions 2 and 3]
{hyeon_predictions}.

\subsection*{GIT stability of ribbons}
Applying Proposition \ref{proposition-git}, we obtain the following result as a corollary of computations
made in Example \ref{E:ribbons}, whose notation we keep. 
\begin{thm}[Hilbert stability of ribbons] \label{theorem-ribbons} 
Let $C_\ell$ be a ribbon defined by $f(y)=y^{-\ell}$ for some $\ell\in \{1,\dots, g-2\}$.
Then $C_\ell$ admits a $\GG_m$-action $\rho\co \GG_m\ra \Aut(C_\ell)$ and 
\begin{enumerate}
\item If $\ell\neq (g-1)/2$, then the 
$m^{th}$ Hilbert point of the $n$-canonical embedding of 
$C_{\ell}$ is unstable
for all $m\geq 2$ and $n\geq 1$.
\item If $\ell= (g-1)/2$, then the $m^{th}$ Hilbert point of the $n$-canonical embedding of 
$C_ {\ell}$ is unstable for all $m\geq 2$ and $n\geq 2$.
\item If $\ell= (g-1)/2$, then the $m^{th}$ Hilbert point of the {\em canonical embedding} of $C_ {\ell}$ is strictly semistable with respect to $\rho$ for all $m\geq 2$.
\end{enumerate}
\end{thm}
\begin{proof} 
The ribbon $C_{\ell}$ is obtained by gluing 
$\spec k[x,\epsilon]/(\epsilon^2)$ and
$\spec k[y,\eta]/(\eta^2)$
along open affines 
$\spec k[x,x^{-1},\eps]/(\eps^2)$ and $\spec k[y,y^{-1},\eta]/(\eta^2)$ by
\begin{align*}
x &\mapsto y^{-1}-y^{-\ell-2}\eta, \\
\eps &\mapsto y^{-g-1}\eta,
\end{align*}
We consider the $\GG_m$-action on $C_\ell$ given by 
$t\cdot (x,y,\eps, \eta)=(tx, t^{-1}y, t^{g-\ell}\eps, t^{-\ell-1}\eta)$. 
It induces a one-parameter subgroup $\rho\co 
\GG_m \ra \Aut(C_{\ell})$. By Example \ref{E:ribbons}, the characters of $C_\ell$ are 
\begin{align*}
\chi_\lambda(C_\ell, \rho)=g\bigl(\ell-\frac{g-1}{2}\bigr), \qquad
\chi_\delta(C_\ell, \rho)=(5g-4)\bigl(\ell-\frac{g-1}{2}\bigr).
\end{align*}
It follows by Proposition \ref{proposition-git} that the Hilbert-Mumford index with respect to $\rho$ 
of the $m^{th}$ Hilbert point of the canonical embedding of $C_{\ell}$ is 
\[
\mu^{\overline{\Hilb}_{g,1}^{\, m}} ([C], \rho)= 
g(g+m-gm) \bigl(\ell-\frac{g-1}{2} \bigr).
\] 
In particular, it is $0$ if and only if $\ell=(g-1)/2$. Similarly, we verify that the Hilbert-Mumford index 
$\mu^{\overline{\Hilb}_{g,n}^{\, m}} ([C], \rho)$ of the $m^{th}$ Hilbert point of the 
$n$-canonical
embedding of $C_{\ell}$ is non-zero for every $n, m\geq 2$. This finishes the proof.
 \end{proof}
\begin{cor}
If $g$ is even, then every canonically embedded genus $g$ 
ribbon with a $\GG_m$-action is Hilbert unstable. 
\end{cor}

\bibliography{BibHK}{}
\bibliographystyle{alpha}

\end{document}